\documentclass[a4paper,11pt]{amsart}

\usepackage[left=2.5cm,right=2.5cm,top=3.5cm,bottom=3cm]{geometry}

\usepackage{amsmath,amssymb,amscd,amsfonts}
\usepackage[hyphenbreaks]{breakurl}
\usepackage{mathrsfs}
\usepackage{latexsym}
\usepackage[all]{xy}
\usepackage{verbatim}
\usepackage[usenames, dvipsnames]{color}
\usepackage{multirow}
\usepackage{enumerate}
\usepackage{url}
\usepackage{mathdots}
\usepackage{MnSymbol}
\usepackage{stmaryrd}
\allowdisplaybreaks
\usepackage[utf8]{inputenc}

\renewcommand{\arraystretch}{1.2}

\usepackage{hyperref}
\hypersetup{colorlinks,breaklinks,
           linkcolor=MidnightBlue,urlcolor=MidnightBlue,
                       anchorcolor=MidnightBlue,citecolor=MidnightBlue}

\newcommand{\N}{\mathbb{N}}
\newcommand{\Z}{\mathbb{Z}}
\newcommand{\C}{\mathbb{C}}
\newcommand{\F}{\mathbb{F}}
\renewcommand{\P}{\mathbb{P}}
\newcommand{\G}{\Gamma} 
\newcommand{\g}{\gamma}
\renewcommand{\t}{\tau}
\renewcommand{\d}{\delta}
\renewcommand{\a}{\alpha}

\newcommand{\m}{\mathfrak{m}}
\renewcommand{\O}{\Omega}
\newcommand{\mf}{\,|_{k,m}} 
\newcommand{\mmf}{\,\,||_{k,m}} 

\newcommand{\cl}{\mathcal}
\newcommand{\p}{\mathfrak{p}}

\newcommand{\tl}[1]{\widetilde{#1}}
\newcommand{\AP}{\mathcal{AP}}

\DeclareMathOperator{\GL}{GL}
\DeclareMathOperator{\SL}{SL}
\DeclareMathOperator{\Ker}{Ker}

\newtheorem{thm}{Theorem}[section]
\newtheorem{prop}[thm]{Proposition}

\newtheorem{lem}[thm]{Lemma}
\newtheorem{cor}[thm]{Corollary}
\newtheorem{defin}[thm]{Definition}
\newtheorem{rem}[thm]{Remark}

\newtheorem{exe}[thm]{Example}

\newcommand{\smatrix}[4]{\left(\begin{smallmatrix} {#1} & {#2} \\ {#3} & {#4} \end{smallmatrix}\right)}

\title[Drinfeld quasi-modular forms]{Drinfeld quasi-modular forms of higher level}

\author{Andrea Bandini}
\address{{\sc Andrea Bandini}: Universit\`a di Pisa \\
Dipartimento di Matematica \\
Largo Bruno Pontecorvo 5 \\
56127 Pisa, Italy }
\email{andrea.bandini@unipi.it}

\author{Maria Valentino}
\address{{\sc Maria Valentino}: Universit\`a della Calabria \\
Dipartimento di Matematica e Informatica \\
Via P. Bucci Cubo 30B VII piano \\
87036 Arcavacata di Rende (CS), Italy }
\email{maria.valentino@unical.it}

\author{Sjoerd de Vries}
\address{{\sc Sjoerd de Vries}: Stockholms Universitet \\
Matematiska Institutionen \\
106 91 Stockholm, Sweden }
\email{sjoerd.devries@math.su.se}

\begin{document}

\keywords{Drinfeld quasi-modular forms, Eisenstein series, Hecke operators, eigenforms, hyperderivatives}

\subjclass{11F52, 11F25}

\begin{abstract} We study the structure of the vector space of Drinfeld quasi-modular forms for congruence subgroups. 
We provide representations as polynomials in the false Eisenstein 
series with coefficients in the space of Drinfeld modular forms
(the $E$-expansion), and, whenever possible, as sums of hyperderivatives of Drinfeld modular forms. 

\noindent
Moreover, we introduce and study the double-slash operator, and use it to provide a well-posed definition for Hecke operators on Drinfeld quasi-modular forms. We characterize eigenforms and, for the special case of Hecke congruence subgroups $\Gamma_0(\mathfrak n)$, we give explicit formulas for the Hecke action on $E$-expansions. 
\end{abstract}

\maketitle

\section{Introduction}

Let $f=\sum_{n\geqslant 0}a_nq^n$, with $q=e^{2\pi i z}$ and $z\in \mathbb{H}$, 
be a modular form of weight $k>0$ for the full modular group $\SL_2(\Z)$. 
It is a well-known fact that the derivative of~$f$, 
\[ D f := \frac{1}{2\pi i} \frac{d f}{dz}=q\frac{df}{dq}=
\sum_{n\geqslant 1} n a_n q^n, \]
fails to be a  modular form in general. 

\noindent
In order to deal with this problem, different strategies have been 
adopted (see~\cite{Z} for an overview). 
One option is to change the differentiation operator so that it preserves 
modularity. For example, the Serre derivative 
\[  \vartheta_k f:= D f -\frac{k}{12}E_2 f, \]
where $E_2$ is the non-modular Eistenstein series of weight 2, transforms $f$
into a modular form of weight $k+2$.

\noindent 
Another option, the one  we are interested in, is to relax the 
conditions on the functional equations a modular form is required to 
satisfy, leading to the notion of {\em quasi-modular forms}. 
More precisely, given two positive integers $k$ and $\ell$, a 
quasi-modular form of weight $k$ and depth $\ell$ for $\SL_2(\Z)$ is 
a holomorphic map $f:\mathbb{H}\longrightarrow \C$ such that there 
exist functions $f_0,\dots,f_\ell$ (with $f_\ell\neq 0$), holomorphic 
both on $\mathbb{H}$ and at infinity, which satisfy the following 
condition:
\begin{equation}\label{e:Intro}
f\left( \frac{az+b}{cz+d}\right)=
(cz+d)^k\sum_{i=0}^\ell f_i(z)\left( \frac{c}{cz+d}\right)^i 
\end{equation}
for all $\smatrix{a}{b}{c}{d}\in\SL_2(\Z)$ and all $ z\in\mathbb{H}$.

The present paper deals with the function field counterpart of 
this framework. Specifically, we recall that {\em Drinfeld modular 
forms} were introduced by 
Goss \cite{G1} and then extensively investigated by Gekeler 
and Goss himself (see \cite{GDRG} and \cite{G2}). We shall mainly work with
{\em Drinfeld quasi-modular forms}, which, on the other hand, were 
first studied by Bosser and Pellarin in \cite{BP} for $\GL_2(\F_q[T])$ 
(the analogue of $\SL_2(\Z)$ in our positive characteristic setting), 
where $q=p^r$ with $p\in\Z$ prime and $r\in \N_{>0}$.

 Our goal here is twofold. First, we provide algebraic structure theorems 
for Drinfeld quasi-modular forms for congruence subgroups $\Gamma$ of 
$\GL_2(\F_q[T])$. We shall represent 
quasi-modular forms via their {\em $E$-expansion}, i.e., as polynomials in the false Eisenstein series
$E$ with modular coefficients (Section \ref{s:EExp}) and, whenever possible, as sums of hyperderivatives of modular forms (Section~\ref{SecHyp}).

The second main goal of this paper is the definition of Hecke operators (Section~\ref{SecHecke}). The first hint at a 
definition of Hecke operators on Drinfeld quasi-modular forms
appeared in~\cite[\S 4.1.1]{BP2}, but it was unclear whether 
it was independent of the chosen set of representatives or not.
We provide a definition which does not depend on the set
of representatives and characterize the resulting eigenforms.
In the case $\G=\G_0(\m)$ (a Hecke congruence subgroup), we also provide explicit formulas for the Hecke action
on the $E$-expansion of a quasi-modular form.

The main ingredient in the proofs is the {\em associated polynomial} $P_f$ of a quasi-modular form $f$
(Section \ref{s:AssPol}).
To describe the set of such polynomials, we introduce the
{\em double-slash operator} (Section \ref{s:DSOp}), adapting the classical definition
of \cite{CL} and \cite{M} to the function field setting. The double-slash operator is crucial in the definition of Hecke operators.

Section~\ref{SecQMod}, in which we define the main objects of study, is written for a general global function field~$K$ and its subring $A$ of functions regular outside a fixed prime~$\infty$. In the other sections we focus on the case
$K=\F_q(T)$ with $\infty=\frac{1}{T}$ and $A=\F_q[T]$.
We expect no major difficulties in extending our main results to a general global function field.

Before giving more details on the main results of this paper,
we remark that, as usual, quasi-modular forms are defined by
algebraic conditions (symmetries with respect to $\G$) and
analytic conditions (holomorphicity at cusps). Some of our results do not require holomorphicity and hold for what we
call quasi-modular \emph{functions}. In this Introduction,
for simplicity, we only provide statements for quasi-modular forms, leaving the general setting of quasi-modular functions to the main body of the paper.

\subsection{Algebraic structures}
Let $\tl{M}_{k,m}^{\,\leqslant \ell}(\Gamma)$ denote the space of 
Drinfeld quasi-modular forms of weight $k$, type $m$ and depth at most 
$\ell$ for $\Gamma$. We shall provide two structure theorems of different nature. The first uses the false
Eisenstein series $E$ (see Example \ref{E:falseEis}) of weight 2, type 1 
and depth 1
to write Drinfeld quasi-modular forms as polynomials in~$E$.

\begin{thm}[Theorem \ref{t:StrutPolE}]\label{t:Intro1}
Every Drinfeld quasi-modular form $f$ of depth $\ell$ for $\Gamma$ 
can be written uniquely as a polynomial in $E$ of degree $\ell$, 
with coefficients in the space of Drinfeld modular forms 
(the {\em $E$-expansion of $f$}). 
\end{thm}

The second structure theorem uses the hyperderivatives
$\cl{D}_n: \tl{M}_{k,m}^{\,\leqslant \ell}(\Gamma)\longrightarrow 
\tl{M}_{k+2n,m+n}^{\,\leqslant \ell+n}(\Gamma)$, thus providing
an analogue of the main result of Kaneko and Zagier in \cite{KZ}. 
Its statement involves an additional
hypothesis on binomial coefficients, which is due to
the very definition of the operators $\cl{D}_n$ and is 
unavoidable. Let $M_{k,m}(\Gamma)$ denote the space of Drinfeld modular forms of weight~$k$ and type~$m$ for~$\Gamma$.
Whenever $k>2\ell$ (respectively, $k=2\ell$), the hypothesis states:
\begin{equation}\label{e:Intro1} 
{{k-i-1} \choose {i}}\not\equiv 0\!\!\!\pmod{p}\ \ \text{for all }\  
1\leqslant i\leqslant \ell\ \text{such that}\ M_{k-2i,m-i}(\G)\neq 0
\end{equation}
(respectively, for all $1\leqslant i\leqslant \ell-1$ such that
$M_{k-2i,m-i}(\G)\neq 0$).
   
\begin{thm}[Theorem \ref{Structure2}]\label{t:Intro2}
Every Drinfeld quasi-modular form in $\tl{M}_{k,m}^{\,\leqslant \ell}(\G)$
(with $\ell>0$) can be written uniquely as a sum of hyperderivatives 
of Drinfeld modular forms and of a scalar multiple of $\cl{D}_{\ell-1}E$ if 
and only if \eqref{e:Intro1} holds.
\end{thm}

Note the contrast with the characteristic 0 setting, where a quasi-modular form can always be written as a sum of derivatives of modular forms. This raises interesting questions
for the cases in which \eqref{e:Intro1} does not hold (see
also Remark \ref{r:Hecke+Der}).

The relation between the two structures (when available) is detailed 
in Section \ref{s:RelStr} using another fundamental object: the 
{\em associated polynomial} $P_f$ of a quasi-modular form $f$ (Definition \ref{d:AssPol}).
The associated polynomial of $f$ is unique as well, hence there 
are bijections between quasi-modular forms, associated polynomials 
and $E$-expansions. The switch between the last two provides a 
$\C_\infty[E]$-algebra isomorphism between associated polynomials in 
the variable $X$ with quasi-modular coefficients and polynomials in $E$ 
with modular coefficients (Theorem \ref{t:CinftyEIso}).

\subsection{Hecke operators}
Let $f\in \tl{M}^{\,\leqslant \ell}_{k,m}(\G)$ be a
quasi-modular form verifying
\begin{equation}\label{e:Intro2} 
f(\g z)=(\det \g)^{-m}(cz+d)^k\sum_{i=0}^\ell f_i(z)
\left(\frac{c}{cz+d}\right)^i 
\end{equation}
for all $\g=\smatrix{a}{b}{c}{d}\in \G$ and all $z\in \Omega$ (the Drinfeld upper-half plane). Then we define the
{\em double-slash operator} as
\[ (f \mmf \gamma)(z) = \sum_{i=0}^{\ell} \left( \frac{-c}{cz+d} \right)^ i 
 (\det\g)^{m-i}(cz+d)^{2i-k}f_i(\g z) .\]
Let $\eta\in \GL_2(\F_q(T))$. The 
 $\eta$-Hecke operator is usually defined in terms of representatives for the orbits of $\G$ acting on the double coset $\G\eta\G$. Using 
 the double-slash operator,
 we similarly obtain a well-defined operator on $\tl{M}^{\,\leqslant \ell}_{k,m}(\G)$ as
 \[ T_\eta(f) := (\det \eta)^{k-m} \sum_\g f\mmf \g ,\]
where $\g$ varies in a set of representatives for $\G\backslash\G\eta\G$.

\noindent
We prove that this definition is independent of the choice of representatives by exploiting the basic properties of the double-slash operator. Moreover, by extending the double-slash operator
to the associated polynomials it is easy to see that

\begin{thm}[Corollary \ref{c:EigenDer}]
A quasi-modular form $f$ verifying \eqref{e:Intro2} is a
$T_\eta$-eigenform of eigenvalue $\lambda$ if and only if, for
all $i=0,\dots,\ell$, the $f_i$ are $T_\eta$-eigenforms of eigenvalue $\frac{\lambda}{(\det \eta)^i}$.
\end{thm} 

For the case $\G=\G_0(\m\p)$ (with $\m$ an ideal of $\F_q[T]$ not divisible by a prime ideal $\p=(\wp)$), we have an explicit set of representatives for $\G\smatrix{1}{0}{0}{\wp}\G$. We use this to compute the action of the corresponding Hecke 
operator (usually denoted by $U_\p$ in this case)
on the $E$-expansion of a quasi-modular form $f$.

\begin{thm}[Corollary \ref{c:UpQuasiMod}]
Let $f=\sum_{i=0}^\ell f_{i,E} E^i\in \tl{M}^{\,\leqslant \ell}_{k,m}(\G_0(\m\p))$ be the $E$-expansion of $f$. Then
\[ U_\p(f)= \sum_{i=0}^\ell \wp^i \left( \sum_{h=0}^{\ell-i} {{h+i} \choose i} U_\p(f_{h+i,E}E_\p^h)\right) E^i\,,\]
where $E_\p(z)=E(z)-\wp E(\wp z)$ is a modular form (the {\em $\p$-stabilization} of $E$). 
\end{thm}

 \medskip
  
At the end of this Introduction we want to mention that, during 
the preparation of this paper, we became aware of the work of Chen and 
Gezmis \cite{CG}\footnote{Recently generalized to congruence subgroups of 
$\GL_2(A)$ in \cite{GV}.}.
They introduce the space of {\em nearly holomorphic Drinfeld modular forms} 
for congruence subgroups $\Gamma \subset \GL_2(\F_q[T])$, which turns out to 
be isomorphic to the space of Drinfeld quasi-modular forms.
The primary interest of \cite{CG} is to study Maass-Shimura operators and 
special values of nearly holomorphic Drinfeld modular forms at CM points. 
In this paper we are mainly interested in providing different algebraic
structures for $\tl{M}_{k,m}^{\,\leqslant \ell}(\Gamma)$ and in 
defining Hecke operators. 
Our approach is different and more direct (one might 
say ``less analytic and more computational''), but obviously recovers 
part of the same structure. In particular,  
our Theorem \ref{t:Intro1} 
corresponds to their Proposition 5.0.12 (which can be considered 
an immediate consequence of their Theorem 3.2.18).

 \section{Drinfeld quasi-modular forms and associated polynomials}\label{SecQMod}
Let $K$ be a global function field with constant field $\F_q$, where $q$ is a power of a prime $p$ in $\Z$. Fix a prime $\infty$ of $K$ and let $A$ be the subring of functions regular outside $\infty$ (the analogue of the ring of integers $\Z$ for $\mathbb{Q}$). Let $d$ be the degree of $\infty$ over $\F_q$. 
The normalized $\infty$-adic (non-archimedean) norm on $A$ is defined as $|a|_\infty := q^{\deg(a)}$ where $\deg(a):=
 d \cdot \text{(the order of the pole of $a$ at $\infty$)}$,
and we extend it canonically to $K$ (so that, in particular,
$|a|_\infty=\left|A/(a)\right|$ for all $a\in A$). 
Let $K_\infty$ be  the completion of $K$ at $\infty$ and denote 
by $\C_\infty$ the completion of an algebraic closure of $K_\infty$.
 
 The set $\O:=\P^1(\C_\infty)-\P^1(K_\infty)$, together with a structure of rigid analytic space (see \cite{FdP}), is the {\em Drinfeld upper half-plane}. 

\noindent The group $\GL_2(K_\infty)$ acts on $\O$ via the usual 
fractional linear transformations $\smatrix{a}{b}{c}{d} (z)=
\frac{az+b}{cz+d}$.
For any $\g\in \GL_2(K_\infty)$, $k,m\in\Z$ and $f: \O\longrightarrow \C_\infty$,
we define the {\em slash operator} $\mf$ by 
\begin{equation}\label{slash}  
(f \mf \g)(z):= (\det \g)^m(cz+d)^{-k} f (\g z). 
\end{equation}
For any nonzero ideal $\m$ in $A$, we define the groups
\begin{align*}
\G(\m)&:=\left\{ \begin{pmatrix} a & b \\ c  & d  \end{pmatrix}\in \GL_2(A)\,:\, a\equiv d \equiv 1 \pmod{\m}, \ b \equiv d \equiv 0 \pmod{\m}\right\},\\
\G_0(\m)&:=\left\{ \begin{pmatrix} a & b \\ c  & d  \end{pmatrix}\in \GL_2(A)\,:\, c\equiv 0\pmod{\m}  \right\} .
\end{align*}
By definition, $\G(1) = \G_0(1)=\GL_2(A)$. A subgroup $\G \subset \GL_2(K)$ is called an \emph{arithmetic subgroup of congruence type} if it is commensurable with $\GL_2(A)$ and contains some $\G(\m)$. Note that if $\g \in \GL_2(K)$ and $\G$ is an arithmetic subgroup of congruence type, then so is $\g^{-1}\G\g$. Moreover, $\det \G$ is a subgroup of $\mathbb F_q^\times$, and we write $o(\G) := \# \det \G$.

Let $\cl{O}$ denote the set of rigid analytic functions $f: \O \longrightarrow \C_\infty$ for which there exists a nonzero ideal 
$I=I(f)$ of $A$ such that $f$ is $I$-periodic, that is, $f(z+a) = f(z)$ for all $a \in I$.

\begin{defin}\label{DefWeakQuasiMod}
Let $k$ and $\ell$ be nonnegative integers, let $m\in \Z$, and let $\G$ be an arithmetic subgroup of congruence type. A rigid analytic function 
$f: \O\longrightarrow \C_\infty$ is called a {\em Drinfeld 
quasi-modular function of weight $k$, type $m$ and depth $\ell$ for $\G$} 
if there exist functions $f_0,\dots,f_\ell\in\cl{O}$
with $f_\ell\neq 0$, 
such that 
\begin{equation}\label{EqDefWeakQuasi} 
(f \mf \g)(z)= \sum_{i=0}^\ell f_i(z)\left( \frac{c}{cz+d}\right)^i
\end{equation}
for all $z\in \O$ and all $\g=\smatrix{a}{b}{c}{d}\in \G$. By convention, $0$ is a Drinfeld quasi-modular function of any weight and type, and depth $\ell = 0$.\\ 
\noindent The set of all Drinfeld quasi-modular functions of weight~$k$, type~$m$, and depth~$\ell$ for~$\G$ is denoted by $\tl{W}_{k,m}^{\,\ell}(\G)$. The $\mathbb C_\infty$-vector space of all quasi-modular functions of weight~$k$, type~$m$, and depth at most~$\ell$ for~$\Gamma$ is denoted by $\tl{W}_{k,m}^{\,\leqslant\ell}(\G):=
\bigcup_{i=0}^\ell \tl{W}_{k,m}^{\,i}(\G)$. We also write $W_{k,m}(\G) := \tl{W}^0_{k,m}(\G)$.
\end{defin}

The vector space $W_{k,m}(\G)$ recovers the \emph{Drinfeld modular functions of level~$\Gamma$} (by which we mean rigid-analytic functions on $\Omega$ that transform like Drinfeld modular forms, but that may not be holomorphic at the cusps; they are also known as \emph{weakly modular forms}). For an example of depth~1, see Example~\ref{E:falseEis} below.

Note that
\eqref{slash} and~\eqref{EqDefWeakQuasi} yield
\begin{equation}\label{e:fgz+Pol} 
f(\g z)=(\det \g)^{-m}(cz+d)^k\sum_{i=0}^\ell f_i(z)
\left(\frac{c}{cz+d}\right)^i.
\end{equation}

\subsection{The associated polynomial}\label{s:AssPol}
We begin by proving the uniqueness of the functions appearing in 
the definition. 

\begin{prop}\label{p:QModPolUn}
Let $f \in \tl{W}_{k,m}^{\,\ell}(\G)$ be a quasi-modular function 
verifying equation~\eqref{EqDefWeakQuasi} for some 
$f_0,\dots,f_\ell\in \cl{O}$. If $f\neq 0$, then the weight $k$, 
the depth $\ell$ and the polynomial 
\[ P_f:=\sum_{i=0}^\ell f_i X^i\in \cl{O}[X] \] 
are uniquely determined by~$f$. Moreover, the type $m$ is unique 
modulo $o(\G)$.
\end{prop}

\begin{proof}
Let $\mathfrak{n}$ be such that $\G(\mathfrak{n})\subseteq \G$.
Assume  that both the polynomials $\sum_{i=0}^\ell f_i X^i$ and 
$\sum_{i=0}^{\ell'} g_i X^i$ are associated to $f$ and that $f$ is also of 
weight $k'$ and type $m'$.
By definition, for any $\g=\smatrix{a}{b}{c}{d}\in\G$ we have 
\begin{align}\label{EqSlash1}
f(\g z) & = (\det \g)^{-m} (cz+d)^k \sum_{i=0}^\ell f_i(z)\left( \frac{c}{cz+d}\right)^i\\ \nonumber 
 & = (\det \g)^{-m'} (cz+d)^{k'} \sum_{i=0}^{\ell'} g_i(z)\left( \frac{c}{cz+d}\right)^i .
\end{align}
Since $\smatrix{1}{0}{0}{1}\in \G$, we have $f_0(z)=g_0(z)=f(z)$ for all 
$z\in \O$.

\noindent
Now, for any $\a\in \mathfrak{n}$, consider the matrix 
$\smatrix{\a^2+1}{\a}{\a}{1}\in\G(\mathfrak{n})\subseteq \G$. By \eqref{EqSlash1} we have
\[ (\a z+1)^k\sum_{i=0}^\ell f_i(z)\left(\frac{\a}{\a z+1} \right)^i= (\a z+1)^{k'}\sum_{i=0}^{\ell'} g_i(z)\left(\frac{\a}{\a z+1} \right)^i .\]
Without loss of generality we assume $k\geqslant k'$ and get
\begin{equation} \label{RedStar}
(\a z+1)^{k-k'}\sum_{i=0}^\ell f_i(z)\left(\frac{\a}{\a z+1} \right)^i = 
 \sum_{i=0}^{\ell'} g_i(z)\left( \frac{\a}{\a z+1} \right)^i .  
\end{equation}
For any fixed $z\in \O$, define the polynomials
\[ C_z(X):=\sum_{i=0}^{\ell} f_i(z) X^i\,,\ D_z(X):= ( 1-z X)^{k-k'}\sum_{i=0}^{\ell'} g_i(z) X^i  \in \C_\infty[X] . \]
By \eqref{RedStar}, as $\alpha$ varies, there exist infinitely many maps 
$X\longmapsto \frac{\a}{\a z+1}$ such that the specializations
of $C_z(X)$ and $D_z(X)$ coincide. Therefore, for all $z\in \O$, we have 
that $C_z(X)-D_z(X)$ is the zero polynomial. By the identity principle, 
all coefficients have to be equal. In particular, looking at the leading
coefficients, we have
 \[ f_\ell(z)=(-z)^{k-k'}g_{\ell'}(z)\quad\text{for all } z\in \O.\]
Hence, $f_\ell(z)$ and $(-z)^{k-k'}g_{\ell'}(z)$ define the same function in 
$\cl{O}$. If $f_{\ell}(z)$ is $I_{f_\ell}$-periodic and $g_{\ell'}(z)$ is $I_{g_{\ell'}}$-periodic, then both $f_{\ell}(z)$ and $g_{\ell'}(z)$ are $I_{f_\ell}I_{g_{\ell'}}$-periodic. 
Since $-z$ is not, it follows that $k=k'$. 
Moreover, the equality of degrees yields $\ell=\ell'$ and 
$f_i(z)=g_i(z)$ for all $i=0,\dots,\ell$.

\noindent
Finally, Equation \eqref{EqSlash1} with a matrix $\g \in \G$ such that $\det \g$ generates $\det \G\subset \F_q^*$ 
yields $m\equiv m'\pmod{ o(\G) }$. 
\end{proof}

Thanks to the previous proposition we can give the following

\begin{defin}\label{d:AssPol}
For any Drinfeld quasi-modular function $f\in \tl{W}_{k,m}^{\,\ell}(\G)$ 
verifying equation~\eqref{EqDefWeakQuasi} with $f_{\ell} \neq 0$, the {\em associated polynomial} 
of~$f$ is
\[ P_f:=\sum_{i=0}^\ell f_i X^i \in \cl{O}[X]. \]
If $f=0$, we set $P_f=0$. The subset of $\mathcal O[X]$ given by all associated polynomials of elements of $\tl{W}_{k,m}^{\,\leqslant \ell}(\G)$ is denoted by $\mathcal{AP}_{k,m}^{\,\leqslant \ell}(\G)$.
\end{defin}

\begin{rem} \label{r:AssPol}
Since the coefficient $f_0$ of $P_f$ is equal to~$f$, it
determines the whole associated polynomial. Hence, in general,
the set of associated polynomials is a proper subset 
of~$\mathcal{O}[X]$. A precise description of this set will 
be given in Lemma \ref{l:AssIFFQModCoeff} and Theorem \ref{t:APTildePEIso} below
after introducing the {\em double-slash operator} on polynomials
(see Equation \eqref{e:DoubleOnPoly}). 
\end{rem} 

\noindent
Take $f\in \tl{W}^{\,\ell_1}_{k_1,m_1}(\G)$, 
$g\in \tl{W}^{\,\ell_2}_{k_2,m_2}(\G)$ and $\lambda\in \C_\infty$. Then, (see \cite[Remark (iv), page 11]{BP})
\begin{enumerate}[(i)]
\item $fg\in \tl{W}^{\,\ell_1+\ell_2}_{k_1+k_2,m_1+m_2}(\G)$ with 
$P_{fg}=P_f P_g$;
\item $P_{\lambda f}=\lambda P_f$;
\item if $k_1=k_2$ and $m_1\equiv m_2 \pmod{o(\G)}$, then 
$f+g\in \tl{W}_{k_1,m_1}^{\,\leqslant\max\{\ell_1,\ell_2\}}(\G)$ with associated polynomial $P_{f+g}=P_f+P_g$.
\end{enumerate}
\smallskip

To fix notations, from now on, for any quasi-modular function $g$, we denote by $g_i$ the coefficient of $X^i$ in its associated polynomial $P_g$. 

\begin{lem}\label{FrenchLemma}
Let $f\in \tl{W}_{k,m}^{\,\leqslant \ell}(\G)$, and let 
$P_f=\sum_{i=0}^\ell f_i X^i$ be its associated polynomial. 
Then,
\[ f_i\in \tl{W}^{\,\leqslant \ell-i}_{k-2i,m-i}(\G) \ \text{ with } \ 
P_{f_i}=\sum_{j=i}^\ell {j \choose i} f_j X^{j-i} \ 
\text{ for all }\ 0\leqslant i\leqslant \ell . \]
In particular, 
\begin{equation}\label{e:AssPolfi}
(f_i)_h={{h+i}\choose i} f_{h+i}\ \text{ for all }\,i=0,\dots,\ell\
\text{ and all }\,h=0,\dots,\ell-i .
\end{equation}

\end{lem}

\begin{proof}
The proof of \cite[Lemma 119]{MR} does not depend on the arithmetic subgroup
and works in our setting as well. One can see also \cite[Lemma 2.5]{BP}. \end{proof}

\subsection{The double-slash operator}\label{s:DSOp}

Although quasi-modular functions are defined in terms of the slash operator, there is in fact a similar, but more natural operator in the present setting. This is the double-slash operator, which appears in the characteristic zero setting in \cite{M} and \cite[Sections 7.3 and~7.4]{CL}. As we will see, it plays the role for quasi-modular functions that the slash operator plays for modular functions.

\begin{defin}\label{DefDoubleHecke}
Let $f\in  \tl{W}^{\,\leqslant \ell}_{k,m} (\G)$ with associated polynomial $P_f = \sum_{i=0}^{\ell} f_i X^i$. The  
{\em double-slash operator} is defined as
\begin{align}\label{eq:def}
(f \mmf \gamma)(z) & = \sum_{i=0}^{\ell} \left( \frac{-c}{cz+d} \right)^ i (f_i\,\,|_{k-2i,m-i}\gamma)(z) \\
 & = \sum_{i=0}^{\ell} \left( \frac{-c}{cz+d} \right)^ i 
 (\det\g)^{m-i}(cz+d)^{2i-k}f_i(\g z) \nonumber
\end{align}
for all $\gamma=\smatrix{a}{b}{c}{d}\in \GL_2(K_\infty)$ and
all $z\in \O$.
\end{defin}

\begin{rem}\label{r:Double1}
Note that we recover the usual slash operator when $\ell=0$ or when operating with a matrix $\g$ for which $c=0$.
  \end{rem}

A priori, $f \mmf \g$ is merely a rigid-analytic function on $\Omega$. We will see in Proposition~\ref{thm:double-slash-properties} that it is in fact again a quasi-modular function.

Equation~\eqref{eq:def} expresses the double-slash operator in terms of slash operators. It is also possible to do the reverse.

\begin{prop}
    Let $f \in \tl{W}_{k,m}^{\,\leqslant \ell}(\Gamma)$ with associated polynomial $P_f = \sum_{i=0}^{\ell} f_i X^i$ and let $\gamma \in \mathrm{GL}_2(K_\infty)$. Then
    \begin{equation}\label{key_eq}
    (f \mf \gamma)(z) = \sum_{i=0}^{\ell} \left( \frac{c}{cz+d} \right)^i (f_i \,\,||_{k-2i,m-i} \gamma)(z).
    \end{equation}
\end{prop}

\begin{proof}
    We use induction on $\ell$, the base case $\ell=0$ being clear. Suppose we know the statement for quasi-modular functions of depth at most $\ell-1$. Then, for $j = 1,\ldots,\ell$, Lemma~\ref{FrenchLemma} and the induction hypothesis
    yield
    \[  (f_j \,|_{k-2j,m-j} \gamma)(z) = \sum_{i=j}^{\ell} \binom{i}{j} \left( \frac{c}{cz+d} \right)^{i-j} (f_i \,\,||_{k-2i,m-i} \gamma)(z). \]
    By \eqref{eq:def} and the fact that $f=f_0$, we get
    \begin{align*}
        (f \mmf \gamma)(z) - (f \mf \gamma)(z) &= \sum_{j=1}^{\ell} \left(\frac{-c}{cz+d}\right)^j (f_j \,|_{k-2j,m-j}\gamma)(z) \\ 
        &= \sum_{j=1}^{\ell} \left( \frac{-c}{cz+d}\right)^j \sum_{i=j}^{\ell} \binom{i}{j} \left( \frac{c}{cz+d} \right)^{i-j} (f_i \,\,||_{k-2i,m-i} \gamma)(z) \\
        &= \sum_{i=1}^{\ell} \left( \sum_{j=1}^i \binom{i}{j} (-1)^{j} \right)  \left( \frac{c}{cz+d} \right)^i (f_i \,\,||_{k-2i,m-i} \gamma)(z) \\
        &= - \sum_{i=1}^{\ell} \left( \frac{c}{cz+d} \right)^i (f_i \,\,||_{k-2i,m-i} \gamma)(z). \qedhere
    \end{align*}
\end{proof}

\subsection{Quasi-modular polynomials}
We now extend the double-slash operator to the polynomial ring $\mathcal O[X]$, which contains quasi-modular functions via the injective map $f \mapsto P_f$. 

Let $P(z,X)\in \cl{O}[X]$ of degree $\ell$ and let $\g=\smatrix{a}{b}{c}{d}\in \GL_2(K_\infty)$. In analogy with \cite[Chapter~7.2]{CL}, we define the double-slash operator on $P$ as
\begin{equation}\label{e:DoubleOnPoly}
(P \mmf \gamma)(z,X)=(\det\g)^m(c z+d)^{-k}P\left(\g z, \frac{(c z+d)^2}{\det\g} \left( X-\frac{c}{c z+d}\right)\right).
\end{equation}
More explicitly, if $P(z,X)=\sum_{i=0}^\ell \alpha_i(z)X^i$, then
\begin{align}\label{e:DoubleOnPoly2}
(P \mmf \gamma)(z,X) & = (\det\g)^m(c z+d)^{-k}
\sum_{i=0}^\ell \alpha_i(\g z)\frac{(cz+d)^{2i}}{(\det \g)^i}
\left( X-\frac{c}{cz+d}\right)^i \\ \nonumber \ & =
\sum_{i=0}^\ell \alpha_i(\g z)\frac{(cz+d)^{2i-k}}{(\det \g)^{i-m}}
\sum_{h=0}^i{i\choose h}\left(\frac{-c}{cz+d}\right)^{i-h}X^h \\
\nonumber
\ & = \sum_{h=0}^\ell \left[ \sum_{i=h}^\ell {i\choose h} 
(\det \g)^{m-i}(cz+d)^{2i-k}\left(\frac{-c}{cz+d}\right)^{i-h} 
\alpha_i(\g z)\right] X^h. 
\end{align}

\begin{defin}\label{d:QModPoly}
We say that $P(z,X)\in \cl{O}[X]$ of degree $\ell$ is a 
{\em weakly quasi-modular polynomial for $\G$ of weight $k$, type $m$ and 
depth $\ell$} if $(P \mmf \gamma)(z,X)=P(z,X)$ for all 
$\g\in \G$ and $z\in \O$.

\noindent
We denote by $\cl{WP}^{\,\leqslant \ell}_{k,m}(\G)\subseteq \cl{O}[X]$ the 
set of all weakly quasi-modular polynomials for $\G$ of weight $k$, 
type $m$ and depth at most $\ell$.
\end{defin}

Our first goal is to show that $\cl{WP}_{k,m}^{\,\leqslant \ell}(\Gamma) = \cl{AP}_{k,m}^{\,\leqslant \ell}(\Gamma)$, which is the content of the next two lemmas.

\begin{lem}\label{l:AssociateIsQMod}
Let $f\in \tl{W}^{\,\leqslant \ell}_{k,m}(\G)$ with associated polynomial $P_f(z,X)=\sum_{i=0}^{\ell} f_i(z)X^i$. 
Then $P_f(z,X)\in \cl{WP}^{\,\leqslant \ell}_{k,m}(\G)$.
\end{lem}

\begin{proof}
Let $\g \in \G$, then (recalling \eqref{e:DoubleOnPoly2}, \eqref{e:AssPolfi} and \eqref{e:fgz+Pol})
\begin{align*}
(P_f \mmf \gamma)(z,X) & = \sum_{h=0}^\ell \left[ \sum_{i=h}^\ell {i\choose h} 
(\det \g)^{m-i}(cz+d)^{2i-k}\left(\frac{-c}{cz+d}\right)^{i-h} 
f_i(\g z)\right] X^h \\
& = \sum_{h=0}^\ell \left[ \sum_{i=h}^\ell \sum_{j=i}^\ell 
{i\choose h} {j\choose i} (-1)^{i-h} \left(\frac{c}{cz+d}\right)^{j-h} 
f_j(z)\right] X^h \\
& = \sum_{h=0}^\ell  \left[  \sum_{j=0}^\ell {j \choose h} f_j(z) 
\left( \frac{c}{cz+d } \right)^{j-h} 
\sum_{i=h}^j  {j-h \choose i-h} (-1)^{i-h} \right] X^h \\
& = \sum_{h=0}^\ell f_h(z) X^h = P_f(z,X). \qedhere
\end{align*}
\end{proof}

\begin{lem}\label{l:AssIFFQModCoeff}
Let $P(z,X)=\sum_{i=0}^\ell \alpha_i(z)X^i\in \cl{O}[X]$. 
Then $P(z,X)\in \cl{WP}^{\,\leqslant \ell}_{k,m}(\G)$ if and only if
for all $i=0,\dots,\ell$ the functions $\alpha_i$ satisfy 
\[  (\alpha_i \,|_{k-2i,m-i}\g)(z)= 
\sum_{j=i}^{\ell} {j \choose i} \alpha_j(z)
\left( \frac{c}{cz+d}\right)^{j-i} \]
for all $z\in\O$ and $\g\in \G$, i.e., if and only if 
$\alpha_i\in \tl{W}^{\,\leqslant \ell-i}_{k-2i,m-i}(\G)$ with 
associated polynomial 
$P_{\alpha_i}=\sum_{j=i}^\ell {j \choose i} \alpha_j X^{j-i}$.
\end{lem}

\begin{proof} ($\Longrightarrow$) Let $P(z,X)\in \cl{WP}^{\,\leqslant \ell}_{k,m}(\G)$ and fix $\g \in \G$. 

\noindent
By assumption, we have $(P \mmf \gamma^{-1})(\g z,X)=P(\g z,X)$, where
$\g^{-1}=\frac{1}{\det\g}\smatrix{d}{-b}{-c}{a}$. Using \eqref{e:DoubleOnPoly2}, this yields for all 
$i=0,\dots,\ell$ 
\begin{align*}  \alpha_i(\g z) & = \sum_{j=i}^\ell {j \choose i} 
(\det \g^{-1})^{m-j} \left(\frac{-c(\g z)+a}{\det \g}\right)^{2j-k} 
\left(\frac{c}{-c(\g z)+a}\right)^{j-i} \alpha_j(\g^{-1}\g z) \\
\ & = \sum_{j=i}^\ell {j \choose i} 
(\det \g)^{j-m} (cz+d)^{k-2j} 
\left(\frac{c(cz+d)}{\det\g}\right)^{j-i} \alpha_j(z) \\
\ & = \sum_{j=i}^\ell {j \choose i} (\det \g)^{i-m}
(cz+d)^{k-2i}\left(\frac{c}{cz+d}\right)^{j-i} \alpha_j(z) . 
\end{align*}
As a consequence 
\[  
(\alpha_i \,|_{k-2i,m-i}\g)(z) = (\det \g)^{m-i}(cz+d)^{2i-k}\alpha_i(\g z)
=\sum_{j=i}^\ell {j \choose i} \alpha_j(z) 
\left( \frac{c}{c z+d}\right)^{j-i}. 
\]
($\Longleftarrow$) By assumption, $P(z,X) = P_{\alpha_0}(z,X)$
with $\alpha_0\in \tl{W}^{\,\leqslant \ell}_{k,m}(\G)$. Hence this direction follows from Lemma~\ref{l:AssociateIsQMod}.
\end{proof}

As an immediate consequence of Lemma \ref{l:AssIFFQModCoeff}, we obtain

\begin{thm}\label{t:APTildePEIso}
For all weights $k$, types $m$ and depths $\ell$, we have $\AP_{k,m}^{\,\leqslant\ell}(\G)=
\cl{WP}^{\,\leqslant \ell}_{k,m} (\G)$. 

\noindent
Moreover, if $P(z,X) = \sum_{i=0}^{\ell} \alpha_i(z)X^i \in \cl{WP}^{\,\leqslant \ell}_{k,m}(\G)$,
then $P(z,X)=P_{\alpha_0}(z,X)$.
\end{thm}

\subsection{The double-slash operator, again}
If $f \in \tl{W}_{k,m}^{\,\leqslant \ell}(\G)$, it is not obvious that $f \mmf \gamma$ is again a quasi-modular function and therefore composition of double-slash operators does not a priori make sense. This problem does not arise for double-slash operators on $\mathcal O[X]$. The next lemma shows that the double-slash operator is associative, i.e., defines an action of $\text{GL}_2(K_\infty)$ on $\mathcal O[X]$.

\begin{lem}\label{lem:assoc}
    For $\gamma, \gamma' \in \GL_2(K_\infty)$ and $P(z,X) \in \mathcal O[X]$, we have
    \[
    (P \mmf \gamma \gamma')(z,X) = ((P \mmf \gamma) \mmf \gamma')(z,X).
    \]
\end{lem}

\begin{proof}
    For any matrix $\eta = \smatrix{a}{b}{c}{d} \in \text{GL}_2(K_\infty)$, write 
    \[
    j(\eta,z) = cz+d\quad\text{and}\quad \kappa(\eta,z) = \frac{c}{cz+d}.
    \]
    Then it is well known that $j(\gamma \gamma',z) = j(\gamma,\gamma'z)j(\gamma',z)$ and that
    \[
    \kappa(\gamma\gamma',z) = \frac{\det \gamma'}{j(\gamma',z)^2}\, \kappa(\gamma,\gamma'z) + \kappa(\gamma',z),
    \]
    see, e.g., \cite[Equations (1.13) and (1.14)]{CL} taking into
    account that we have nontrivial determinants. The lemma is now verified by expanding both sides of the equation according to the definition~\eqref{e:DoubleOnPoly}. Indeed, one finds that
    \begin{align*}
    (P\mmf \gamma \gamma')(z,X) &= (\det \gamma \gamma')^m j(\gamma \gamma',z)^{-k} P(\gamma \gamma'z, Y_1), \\ ((P\mmf \gamma)\mmf \gamma')(z,X) &= (\det \gamma \gamma')^m j(\gamma \gamma',z)^{-k} P(\gamma \gamma'z, Y_2),
    \end{align*}
    where
    \begin{align*}
        Y_1 &= \frac{j(\gamma \gamma',z)^2}{\det \gamma \gamma'} \left(X - \kappa(\gamma \gamma',z) \right), \\
        Y_2 &= \frac{j(\gamma, \gamma'z)^2}{\det \gamma} \left( \frac{j(\gamma',z)^2}{\det \gamma'} \left(X - \kappa(\gamma',z) \right) - \kappa(\gamma,\gamma'z)\right),
    \end{align*}
    and the relations above readily imply that $Y_1 = Y_2$.
\end{proof}

\begin{prop}\label{thm:double-slash-properties}
    Let $f \in \tl{W}_{k,m}^{\, \leqslant \ell}(\Gamma)$ and $g \in \tl{W}_{k',m'}^{\,\leqslant \ell'}$. Then for any $\eta, \eta' \in \GL_2(K)$, we have
    \begin{enumerate}[{\bf 1.}]
        \item $f \mmf \eta \in \tl{W}_{k,m}^{\, \leqslant \ell}(\eta^{-1}\Gamma \eta)$;
        \item $P_{f \mmf \eta} = P_f \mmf \eta$;
        \item $f \mmf \gamma=f$ for all 
        $\gamma\in \G$; 
        \item $f \mmf (\eta \eta') = (f \mmf \eta) \mmf \eta'$.
        \item $(f\cdot g)\,||_{k+k',m+m'} \gamma =(f\,||_{k,m}\gamma)\cdot( g\,||_{k',m'} \gamma ) \text{ for all }\, \g\in\GL_2(K_\infty)$.
        \item If $P_f = \sum f_i X^i$, then $P_{f \mmf \eta} = \sum (f_i \,||_{k-2i,m-i} \eta)X^i$.
    \end{enumerate}
\end{prop}

\begin{proof}
    Let $\gamma \in \Gamma$. By Lemma~\ref{lem:assoc} and the fact that $P_f \in \cl{WP}^{\,\leqslant \ell}_{k,m}(\G)$,
    we have
    \[
    (P_{f} \mmf \eta) \mmf \eta^{-1}\gamma \eta = (P_f \mmf \gamma) \mmf \eta = P_f \mmf \eta.
    \]
    This shows that $P_{f} \mmf \eta$ is a weakly quasi-modular polynomial for the arithmetic subgroup $\eta^{-1} \Gamma \eta$. By Theorem~\ref{t:APTildePEIso}, $P_f \mmf \eta = P_{g}$ for the quasi-modular function $g = (P_f \mmf \eta)(z,0)$. It follows from \eqref{e:DoubleOnPoly2} that
    \[
    (P_{f} \mmf \eta)(z,0) = (f \mmf \eta)(z),
    \]
    which yields {\bf 1} and {\bf 2}. 
    
    \noindent Combining {\bf 2} with Lemma \ref{l:AssociateIsQMod}, we find that $P_{f \mmf \g} = P_f$ for all $\g \in \G$. Since the association $g \mapsto P_g$ is injective, this implies~{\bf 3}. Similarly, by~{\bf 2} and Lemma~\ref{lem:assoc},
    \[
    P_{f \mmf (\eta \eta')} = P_f \mmf (\eta \eta') = (P_f \mmf \eta) \mmf \eta' = (P_{f \mmf \eta} ) \mmf \eta' = P_{(f \mmf \eta) \mmf \eta'},
    \]
    which yields {\bf 4}.

    \noindent For {\bf 5}, recall that $P_{fg} = P_{f}P_{g}$. Using {\bf 2} and \eqref{e:DoubleOnPoly}, we obtain
    \begin{align*}
    P_{(f\,||_{k,m}\gamma)\cdot( g\,||_{k',m'} \gamma )}(z,X) &= (P_{f} \mmf \g)(z,X)\cdot (P_{g} \,||_{k',m'}\g)(z,X) \\
    &= (\det \g)^{m+m'}(cz+d)^{-(k+k')} P_{fg}\left(\g z, \frac{(c z+d)^2}{\det\g} \left( X-\frac{c}{c z+d}\right)\right) \\
    &= (P_{fg} \,||_{k+k',m+m'} \gamma)(z,X) = P_{fg\,||_{k+k',m+m'} \gamma}(z,X).
    \end{align*}
    For {\bf 6}, the associated polynomial of $f\mmf \eta$ has an explicit description via~\eqref{e:DoubleOnPoly2}. One then sees that the coefficient of $X^i$ equals $f_i\,||_{k-2i,m-i} \eta$ by combining equations \eqref{slash}, \eqref{e:AssPolfi} and~\eqref{eq:def}.
\end{proof}

\subsection{Holomorphicity at cusps}
By definition, any $f \in \mathcal O$ is $I$-periodic for some nonzero ideal $I$ of $A$. By \cite[Section 5]{BBP} (which
treats general ranks, our case corresponds to $r=2$), this condition guarantees the existence of a $t_I$-expansion of $f$, which can be thought of as the Laurent expansion of~$f$ at infinity. If the $t_I$-expansion is a power series rather than merely a Laurent series, we say $f$ is \emph{holomorphic at infinity}. Equivalently, if $f$ is $I$-periodic, then $f$ is holomorphic at infinity if and only if $f$ is \emph{bounded on vertical lines}, i.e., $|f(z)|_\infty$ is bounded as $|z|_i \rightarrow +\infty$, where we write $|z|_i := \inf\{|z-x|_\infty \, : \, x \in K_\infty\}$ for $z \in \O$. Note that this latter property makes sense also for functions $f: \O \to \mathbb C_\infty$ which are not $I$-periodic.

Suppose now that $f \in \tl{W}_{k,m}^{\,\ell}(\G)$ with associated polynomial $\sum_{i=0}^\ell f_i X^i$. Holomorphicity of $f$ at infinity does not yet guarantee holomorphicity of the $f_i$ at infinity,
nor does it guarantee holomorphicity of $f$ at all cusps of $\G$. Note however that $\GL_2(K)$ acts transitively on the cusps. The following proposition shows that one can use either the slash or double-slash operator to test holomorphicity at all cusps. 

\begin{prop}\label{prop:holom}
Let $f \in \tl{W}_{k,m}^{\,\leqslant \ell}(\G)$ with associated polynomial $P_f = \sum f_iX^i$. Then the following statements are equivalent.
\begin{enumerate}[{\bf 1.}]
    \item For all $0 \leqslant i \leqslant \ell$ and any $\g \in \GL_2(K)$, $(f_i \, |_{k-2i,m-i} \g)(z)$ is bounded on vertical lines. 
    \item For all $0 \leqslant i \leqslant \ell$ and any $\g \in \GL_2(K)$, $(f_i \, ||_{k-2i,m-i} \g)(z)$ is holomorphic at infinity.
\end{enumerate}
\end{prop}

\begin{proof}By Proposition~\ref{thm:double-slash-properties}.{\bf 1}, $f \mmf \g$ is a quasi-modular function for $\g^{-1}\G \g$. Since $\g \in \GL_2(K)$ and $\G$ is of congruence type, $\g^{-1}\G \g$ contains $\G(I)$ for some nonzero ideal $I$ of $A$, so in particular $f \mmf \g$ is $I$-periodic.

    ${\bf 1} \implies {\bf 2.}$  Equation~\eqref{eq:def} shows that $f \mmf \g$ is bounded on vertical lines, since also $\kappa(\g,z) = \frac{c}{cz+d}$ is bounded (in fact, tends to zero) on vertical lines. Since the $f_i$ are themselves quasi-modular functions with associated polynomials given by Lemma~\ref{FrenchLemma}, the same argument works for all $0 \leqslant i \leqslant \ell$.
    
    ${\bf 2} \implies {\bf 1.}$ This follows from Equation~\eqref{key_eq} and Lemma~\ref{FrenchLemma}.
\end{proof}

\begin{rem}\label{rem:coset_reps}
    By \cite[Section 6]{BBP} (but using the double-slash operator instead of the slash operator), the equivalent conditions of Proposition~\ref{prop:holom} are determined by finitely many $\gamma$. If $A$ is a principal ideal domain, then these finitely many elements may moreover be chosen to lie in $\GL_2(A)$.
\end{rem}

\begin{defin}
    Let $k$ and $\ell$ be nonnegative integers, and let $m\in \Z$. A rigid analytic function 
$f: \O\longrightarrow \C_\infty$ is called a \emph{Drinfeld quasi-modular 
form of weight $k$, type $m$ and depth $\ell$ for $\G$} if 
$f \in \tl{W}_{k,m}^{\,\ell}(\G)$ and $f$ satisfies the equivalent properties of Proposition~\ref{prop:holom}.

\noindent
We denote the vector space (respectively, the set)
of Drinfeld quasi-modular forms of weight $k$, type $m$ and 
depth at most (respectively, exactly) $\ell$ by
$\tl{M}^{\,\leqslant \ell}_{k,m}(\G)$ (respectively, $\tl{M}^{\,\ell}_{k,m}(\G)$). We also put $\tl{M}_{k,m}(\G):=
\bigcup_{\ell\geqslant 0} \tl{M}^{\,\leqslant \ell}_{k,m}(\G)$.
\end{defin}

\begin{exe}\label{E:falseEis}
A crucial example of a quasi-modular form is the false Eisenstein series defined in \cite[\S 8]{G}: 
\[ E(z)=\tl{\pi}^{\,-1}\sum_{\begin{subarray}{} a\in \F_q[T]\\
a\text{ monic}\end{subarray}} \, \sum_{b\in \F_q[T]}
\frac{a}{az+b} , \]
where $\tl{\pi}$ is a chosen fixed period for the 
classic Carlitz module.

\noindent
The function $E$ verifies (see \cite[(8.4)]{G})
\begin{equation}\label{eq:E} 
E(\g z)= (\det \g)^{-1}(cz+d)^2 \left( E(z)-\frac{c}{\tl{\pi}(cz+d)} \right) 
\end{equation}
for all $\g=\smatrix{a}{b}{c}{d}\in \GL_2(\F_q[T])$ and all
$z\in\O$. Moreover, the $t$-expansion \cite[(8.2)]{G} shows that $E$ is holomorphic at infinity\,\footnote{In fact, $E$ has a (simple) zero at infinity, contrary to usual Eisenstein series.}. Since $\mathbb F_q[T]$ is a principal ideal domain and $E\,||_{2,1}\g = E$ for all $\g \in \GL_2(\mathbb F_q[T])$, Remark~\ref{rem:coset_reps} implies that
\[ E\in \tl{M}_{2,1}^{\,1}(\GL_2(\F_q[T]))\ \text{ and }\ 
P_E=E-\tl{\pi}^{\,-1}X . \]
\end{exe} 

When working with quasi-modular functions, the subspaces of quasi-modular forms are usually preserved. In particular, we have

\begin{lem}\label{lem:forms-preserved}
    Let $f \in \tl{M}_{k,m}^{\,\leqslant \ell}(\G)$ with associated polynomial $P_f = \sum_{i=0}^{\ell} f_i X^i$.
    \begin{enumerate}[{\bf 1.}]
        \item For all $0 \leqslant i \leqslant \ell$, we have $f_i \in \tl{M}_{k-2i,m-i}^{\,\leqslant \ell-i}(\G)$.
        \item For all $\g \in \GL_2(K)$, $f\mmf \gamma \in \tl{M}_{k,m}^{\,\leqslant \ell}(\g^{-1}\G \g)$.
    \end{enumerate}
\end{lem}

\begin{proof}
    {\bf 1} follows immediately from Lemma~\ref{FrenchLemma}. 
    
    \noindent
    For {\bf 2}, note that $f \mmf \g$ has an associated polynomial with coefficients $f_i \, ||_{k-2i,m-i} \gamma$, by Proposition~\ref{thm:double-slash-properties}. Thus it suffices that $(f_i \, ||_{k-2i,m-i} \gamma) \,||_{k-2i,m-i}\delta$ is holomorphic at infinity for all $\delta \in \GL_2(K)$. By Proposition~\ref{thm:double-slash-properties} again, this equals $f_i \,||_{k-2i,m-i} (\gamma \delta)$, which is holomorphic at infinity because $f$ is a quasi-modular form.
\end{proof}

Define the space of \emph{quasi-modular polynomials of weight $k$, type $m$ and depth at most $\ell$ for $\G$} as
\[
\mathcal{P}_{k,m}^{\,\leqslant \ell}(\G) = \left\{ P(z,X) = \sum_{i=0}^{\ell} \alpha_i(z) X^i \in \cl{WP}_{k,m}^{\,\leqslant \ell}(\G) \ \bigg{| }\ \alpha_i \in \tl{M}_{k-2i,m-i}^{\,\leqslant \ell-i}(\G) \ \forall i = 0, \ldots, \ell \right\}.
\]
Then Lemma~\ref{lem:forms-preserved} together with Theorem~\ref{t:APTildePEIso} yields

\begin{cor}\label{cor:restriction}
    For any $k,\ell \geqslant 0$, $m\in \mathbb Z$, and any arithmetic subgroup $\G$ of congruence type, the maps $f \mapsto P_f$ induce a commutative diagram
    \begin{equation}\xymatrix{
\tl{M}^{\,\leqslant \ell}_{k,m}(\G) \ar@{^{(}->}[d] \ar[rr]^{\sim} & & \mathcal{P}^{\,\leqslant \ell}_{k,m}(\G)  \ar@{^{(}->}[d] \\
\tl{W}^{\,\leqslant \ell}_{k,m}(\G) \ar[rr]^{\sim} 
& & \cl{WP}^{\,\leqslant \ell}_{k,m}(\G) \,. }
\end{equation}
\end{cor}

\section{The $E$-expansion}\label{s:EExp}
From here until the end of the paper, we fix $K=\F_q(T)$, $\infty=\frac{1}{T}$, and $A=\F_q[T]$.

\noindent
We move from the associated polynomial to a different representation 
of quasi-modular functions, which involves only modular functions and powers of $E$.

\begin{prop}\label{p:StructChange}
Let $f\in \widetilde{W}^{\,\leqslant \ell}_{k,m}(\G)$ with 
$P_f=\sum_{i=0}^\ell f_iX^i$. Define
 \[ f_{i,E} = (-\tl{\pi})^i \sum_{h=i}^\ell  
{h \choose i}  f_h (\tl{\pi}E)^{h-i}
.\]
Then $f_{i,E}\in W_{k-2i,m-i}(\G)$ for all $i=0,\dots,\ell$ and
\[ f= \sum_{i=0}^\ell  f_{i,E} E^i. \] 
We call the last expression (as well as the tuple 
$\cl{E}_f := (f_{0,E},\ldots,f_{\ell,E})$) the 
{\em $E$-expansion} of~$f$.
\end{prop}

\begin{proof}
By definition of the $f_{i,E}$ and the fact that $f_0=f$ we have
\begin{align*}
\sum_{i=0}^\ell  f_{i,E} E^i & = 
\sum_{i=0}^\ell (-\tl{\pi})^i \sum_{h=i}^\ell  
{h \choose i} f_h (\tl{\pi}E)^{h-i}  E^i  \\
\ & = \sum_{h=0}^\ell f_h (\tl{\pi}E)^h 
\sum_{i=0}^h  {h \choose i} (-1)^i 
= f_0=f.
\end{align*}
By Lemma \ref{FrenchLemma}, the $f_{i,E}$ are clearly quasi-modular functions of weight $k-2i$, type $m-i$ 
and depth $\leqslant \ell-i$. Moreover, we can compute their associated polynomial (using Lemma \ref{FrenchLemma} again). 
Recall that $P_{\tl{\pi}E}=\tl{\pi}E-X$. Thanks to the properties of the associated polynomial we have
\begin{align*}
P_{f_{i,E}} & = (-\tl{\pi})^i \sum_{h=i}^\ell {h \choose i} 
P_{f_h} P_{\tl{\pi}E}^{h-i} =  
(-\tl{\pi})^i \sum_{h=i}^\ell \sum_{j=h}^\ell {h \choose i}
{j\choose h}f_jX^{j-h} \cdot P_{\tl{\pi}E}^{h-i} \\
 & =(-\tl{\pi})^i \sum_{j=i}^\ell {j \choose i} f_j
\sum_{h=i}^j {j-i\choose h-i} X^{j-h} \cdot P_{\tl{\pi}E}^{h-i} \\
&  = (-\tl{\pi})^i \sum_{j=i}^\ell {j \choose i} f_j
(P_{\tl{\pi}E}+X)^{j-i} 
= (-\tl{\pi})^i \sum_{j=i}^\ell {j \choose i} f_j (\tl{\pi}E)^{j-i} =
f_{i,E} .
\end{align*}
Hence $f_{i,E}$ has depth 0, i.e., it is a modular function. 
\end{proof}

The previous proposition yields a decomposition
\begin{equation}\label{e:PolInE} 
\tl{W}_{k,m}^{\,\leqslant \ell}(\G) = 
\sum_{i=0}^\ell W_{k-2i,m-i}(\G) E^i,
\end{equation}
which is actually a direct sum. Indeed, the $(\ell+1)\times(\ell+1)$
 matrix $\Phi_{\ell,P\mapsto \mathcal{E}}$ with coefficients in 
 $\C_\infty[E]$ providing the transformation 
 \[ \Phi_{\ell,P\mapsto \mathcal{E}}\, 
 (f_0,\dots,f_\ell)^t= 
 (f_{0,E},\dots,f_{\ell,E})^t \]
(where $\,^t$ denotes the transpose) is upper triangular with determinant a power of $-\tl{\pi}$, and is obviously invertible. 
By Proposition \ref{p:QModPolUn} and Lemma \ref{FrenchLemma}, 
the quasi-modular functions $f_i$ in the polynomial $P_f$ are
unique. Hence the $f_{i,E}$ are uniquely determined by~$f$ as well and we have proved 
the following.

\begin{thm}\label{t:StrutPolE}
There is a direct sum decomposition
\begin{equation}\label{e:PolInE1} 
\tl{W}_{k,m}^{\,\leqslant \ell}(\G) = 
\bigoplus_{i=0}^\ell W_{k-2i,m-i}(\G) E^i .
\end{equation} 
\end{thm}

\subsection{The $\C_\infty[E]$-algebra structure} 
We provide below the formulas
for the reverse transformation from the $E$-expansion of $f$ to the associated polynomial $P_f$. 

\noindent
We like this computational approach and provide formulas for completeness, 
but for a different proof of the direct sum decomposition in level 1  
the reader can also refer to 
\cite[Lemma 2.6]{BP}. Moreover, we recall that, as mentioned in the Introduction,
the same structure theorem is obtained in \cite[Proposition 5.0.12]{CG}
as a consequence of the structure of nearly holomorphic Drinfeld modular forms
provided in \cite[Theorem 3.2.18]{CG}. 
\smallskip

\noindent
The transformation from $\mathcal{E}_f$ to $P_f$ 
(here, with a little abuse of notation, $P_f$ denotes the vector $(f_0,\dots,f_\ell)$) is represented by the 
$(\ell+1)\times(\ell+1)$ matrix $\Psi_{\ell,\mathcal{E}\mapsto P}=
\Phi^{-1}_{\ell,P\mapsto \mathcal{E}}$ whose coefficients (in $\C_\infty[E]$) 
can be computed with the following

\begin{prop}\label{p:PEPA}
With the above notation, for all $i=0,\dots,\ell$
\[ f_i = (-\tl{\pi})^{-i} \sum_{h=i}^\ell {h\choose i}f_{h,E} E^{h-i}\,.\]
\end{prop}

\begin{proof}
We have
\[ P_f = \sum_{h=0}^\ell P_{f_{h,E}E^h} = \sum_{h=0}^\ell P_{f_{h,E}}(P_E)^h
=\sum_{h=0}^\ell f_{h,E}(E-\tl{\pi}^{\,-1}X)^h ,\]
where the last equality holds because all $f_{h,E}$ have depth 0.

\noindent
Hence
\[ P_f= \sum_{h=0}^\ell f_{h,E} 
\sum_{i=0}^h {h\choose i} E^{h-i}(-\tl{\pi})^{-i} X^i =\sum_{i=0}^\ell 
\left[ \sum_{h=i}^\ell {h\choose i} f_{h,E} E^{h-i}\right] 
(-\tl{\pi})^{-i} X^i .\]
By the uniqueness of Proposition \ref{p:QModPolUn}, we get the equality of
the coefficients of $X^i$ in the two expressions for $P_f$.
\end{proof}

\begin{rem}\label{r:BPForDer}
    The reverse formula is given by the formal derivative of the $E$-expansion with respect to $E$, i.e.,
    \begin{equation}\label{e:ForDer}
(-\tl{\pi})^i f_i=\frac{1}{i!}\cdot\frac{d^if}{dE^i} 
\end{equation}
as noted in \cite[Lemma~4.3]{BP2}.
    \end{rem}

\begin{thm}\label{t:CinftyEIso}
The maps
\[ \xymatrix{ \displaystyle{\tl{W}_{k,m}(\G)=
\bigoplus_{i\geqslant 0} W_{k-2i,m-i}(\G) E^i}\  
\ar@<-.5ex>[r]_{\ \ \psi}  &   
\ \cl{WP}_{k,m}(\G):=\displaystyle{\bigcup_{\ell\geqslant 0} \cl{WP}^{\,\leqslant \ell}_{k,m} (\G)} 
 \ar@<-.5ex>[l]_{\ \ \varphi} }\] 
\[\xymatrix{  & \hspace{-.7cm}\mathcal{E}_f\  \ar@<-.5ex>[r] & \  P_f \ar@<-.5ex>[l] }
\] 
are mutually inverse isomorphisms of filtered (by depth)
$\C_\infty$-vector spaces.

\noindent
Moreover, they naturally induce mutually inverse isomorphisms of filtered 
and bi-graded (by weight and type) $\C_\infty[E]$-algebras
\[ \xymatrix{ \displaystyle{\tl{W}(\G):=\bigoplus_{k,m} \tl{W}_{k,m}(\G)} \  
\ar@<-.5ex>[r]_{\Psi\ \ }  &   
\ \cl{WP}(\G):=\displaystyle{\bigoplus_{k,m}
\cl{WP}_{k,m}(\G)}
 \ar@<-.5ex>[l]_{\Phi\ \ } } \] 
where the action of $E$ is defined as $E\cdot \mathcal{E}_f:
=\mathcal{E}_{Ef}=(0,\mathcal{E}_f)$ and 
$E\cdot P_f:=P_{Ef}=P_EP_f$.
\end{thm}

\begin{proof} 
By the previous computation and Theorem \ref{t:APTildePEIso},
the matrices $\Psi_{\ell,\mathcal{E}\mapsto P}$ and 
$\Phi_{\ell,P\mapsto \mathcal{E}}$ provide isomorphisms between
$\bigoplus_{i=0}^\ell W_{k-2i,m-i}(\G) E^i$
and $\cl{WP}^{\,\leqslant \ell}_{k,m} (\G)$ (for all $\ell\geqslant 0$).
The general isomorphism follows 
by simply taking the direct limit over~$\ell$ with respect
to the natural inclusion homomorphisms.

\noindent
The last statement follows immediately by taking direct sums over $k$ and~$m$, and 
it is easy to check that the maps are compatible with the filtered 
$\C_\infty[E]$-algebra structure.
\end{proof}

\begin{rem}
    Corollary~\ref{cor:restriction} shows that all results in this section continue to hold after replacing quasi-modular functions by quasi-modular forms and weakly quasi-modular polynomials by quasi-modular polynomials. In particular, the restrictions of $\Phi$ and $\Psi$ induce isomorphisms between $\tl{M}(\G) := \bigoplus_{k,m} \bigoplus_{i \geqslant 0} M_{k-2i,m-i}(\G)E^i$ and $\cl{P}(\G) := \bigoplus_{k,m} \bigcup_{\ell \geqslant 0} \cl{P}_{k,m}^{\,\leqslant \ell}(\G)$.
\end{rem}

\begin{rem}\label{r:kequiv2m}
We recall that weight and type are not independent of each other
for modular forms. Indeed, 
let $s(\G)$ be the number of scalar matrices in $\G$, then
$k\not\equiv 2m \pmod{s(\G)}$ yields $M_{k,m}(\G)=0$.  
Whenever $k\not\equiv 2m \pmod{s(\G)}$ we have 
$M_{k-2i,m-i}(\G)=0$ for all $i$, hence 
$\tl{M}_{k,m}^{\,\leqslant \ell}(\G)=0$ (for all $\ell$) as well. 

\noindent
Moreover, since $M_{k-2i,m-i}(\G)=0$ for 
all negative $k-2i$, we have 
$\tl{M}_{k,m}^{\,\ell}(\G)=0$ for all $\ell>2k$, and the spaces
$\tl{M}_{k,m}^{\,\leqslant \ell}(\G)$ stabilize as $\ell$ grows.
\end{rem}

\section{Hyperderivatives of Drinfeld quasi-modular forms}\label{SecHyp}
In this section we study hyperderivatives on quasi-modular forms for arithmetic subgroups. We prove that hyperderivatives commute with the double-slash operator and prove Theorem~\ref{t:Intro2}.

Let $f:\O\longrightarrow \C_\infty$ be a rigid analytic function and let
$z\in \O$. 
Following \cite[Section 3.1]{BP}, for any $n\geqslant 0$, we define the
{\em hyperderivatives} $(\cl{D}_n f)(z)$ of $f$  at $z$ by the formula
\begin{equation}\label{eq:DefHypDer}
    f(z+\varepsilon)=\sum_{n\geqslant 0} (\cl{D}_n f)(z)\varepsilon^n ,
    \end{equation} 
where $\varepsilon\in\C_\infty$ and $|\varepsilon|_\infty$ is small.
Thanks to \cite{US}, we know that the family of operators $\{\cl{D}_n\}_{n\in\N}$
defines an iterative higher derivation on the $\C_\infty$-algebra
$\cl{R}$ of rigid analytic functions, that is:
\begin{enumerate}[(i)]
\item  $\{ \cl{D}_n\}_{n\in\N}$ is a family of $\C_\infty$-linear maps from 
$\cl{R}$ to itself;
\item $\cl{D}_0$ is the identity map;
\item if $f,g\in\cl{R}$, then 
\[ \cl{D}_i(fg)=\sum_{r=0}^i(\cl{D}_r f)(\cl{D}_{i-r}g);  \]
\item for $f\in\cl{R}$ and all integers $n,k\geqslant 0$ 
\[ \cl{D}_{np^k}(f^{p^k})=(\cl{D}_n f)^{p^k};  \]
\item ({\em iterativity}) for all integers $i,j\geqslant 0$ 
\[  \cl{D}_i\circ \cl{D}_j=\cl{D}_j\circ \cl{D}_i =
{ i+j \choose i}\, \cl{D}_{i+j}. \]
\end{enumerate}
In the formulas we shall always assume $\mathcal{D}_i\equiv 0$ for all
$i<0$.

\noindent
As remarked in \cite[Page 17]{BP} (and proved in \cite[Lemma 3.1]{US}), an important property of the operators $\cl{D}_n$ is that if $f$ is a rigid analytic function, then $\cl{D}_n f$ is rigid analytic as well. 

In the next propositions we collect a few technical results needed 
to compute $\cl{D}_n$ on quasi-modular forms: proofs are based
on computation of formal equalities and then applying the definition \eqref{eq:DefHypDer}. Hence, we refer the reader to  
the proofs in \cite{US} and/or \cite{BP}, which work in our setting as well taking into account the fact that we have nontrivial determinants, and the different notations (in particular
note that $D_n$ in \cite{US} is our $\cl{D}_n$, in accordance with the notation of \cite{BP}). We then provide the full
proof for Proposition \ref{prop:der_commutes_slash} because it involves the double-slash operator.

The first proposition provides the connection between the polynomial $P_f$ of a quasi-modular function $f$ and the polynomial 
$P_{\cl{D}_n f}$. 

\begin{prop}\label{DerAssPol}
Let $f\in \tl{W}_{k,m}^{\,\leqslant \ell}(\G)$ be a quasi-modular 
function and let $P_f=\sum_{i=0}^\ell f_i X^i$
be its associated polynomial. Then, for all $n\geqslant 0$, we have
$\cl{D}_n f\in \tl{W}_{k+2n,m+n}^{\,\leqslant \ell+n}(\G)$ with associated polynomial
\[  P_{\cl{D}_nf}= \sum_{j=0}^{\ell+n} 
\left[ \sum_{h=0}^n {n+k+h-j-1 \choose h} \cl{D}_{n-h} f_{j-h}\right]X^j \]
(with the convention that $f_i=0$ if $i<0$ or $i>\ell$). 
\end{prop}

\begin{proof} 
See \cite[Proposition 3.1]{BP}.
\end{proof}

The second proposition deals with all objects that will appear 
while composing $\cl{D}_n$ with double-slash operators.

\begin{prop}\label{prop:der_f_compose_g}
Let $f:\O\rightarrow \C_\infty$ be a rigid analytic function and 
$\g=\smatrix{a}{b}{c}{d}\in\GL_2(K_\infty)$. Then, for all
$n\in\N$ and $z\in\O$ we have
\begin{enumerate}[{\bf 1.}]
\item $\displaystyle{\cl{D}_n((cz+d)^{-m})=
{-m\choose n}c^n(cz+d)^{-m-n};}$
\item $\displaystyle{  \cl{D}_n(f\circ \g)(z)=
\sum_{j=0}^n {n-1 \choose j} \frac{(-c)^j}{(cz+d)^{2n-j}} 
(\det \g)^{n-j} (\cl{D}_{n-j} f)(\g z);}$
\item $\displaystyle{ \cl{D}_n(f\mf \g) = 
\sum_{j=0}^n {-k-j \choose n-j} 
{\left(\frac{c}{cz+d}\right)^{n-j} 
(\cl{D}_j f)\,|_{k+2j, m+j} \g}.}$ 
\end{enumerate}
\end{prop}

\begin{proof}
Formula {\bf 1} is \cite[Corollary 3.3]{US} (see also 
\cite[last line in page 20]{BP}). 

\noindent Formula {\bf 2} corresponds to \cite[Theorem 3.4]{US}
(see also \cite[Lemma 3.3]{BP}). 

\noindent Formula {\bf 3} is 
\cite[Corollary 3.5]{US}, noting that  
$g\,|_{k+n+j,m+j}\g=(cz+d)^{j-n}g\,|_{k+2j,m+j}\g$.
\end{proof}

\begin{prop}\label{prop:der_commutes_slash}
    Let $\G$ be an arithmetic subgroup of congruence type and let $\g \in \GL_2(K)$. For any $n \in \mathbb N$, we have a commutative diagram
    \begin{equation}\xymatrix{
\tl{W}^{\,\leqslant \ell}_{k,m}(\G) \ar[d]^{\cl{D}_n} \ar[rrr]^{\mmf\g\ \ } & & &\tl{W}^{\,\leqslant \ell}_{k,m}(\g^{-1}\G\g)  \ar[d]^{\cl{D}_n} \\
\tl{W}^{\,\leqslant \ell+n}_{k+2n,m+n}(\G) \ar[rrr]^{||_{k+2n,m+n}\g\ \ \ } 
& & &\tl{W}^{\,\leqslant \ell+n}_{k+2n,m+n}(\g^{-1}\G\g)  }.
\end{equation}
\end{prop}

\begin{proof}
        Let $f \in \tl{W}_{k,m}^{\leqslant\, \ell}(\G)$ with $P_f=\sum_{i=0}^\ell f_iX^i$, then 
        \begin{align*}
        \cl{D}_n & (f||_{k,m}\g)  \stackrel{\eqref{eq:def}}{=} 
        \cl{D}_n \left( \sum_{i=0}^{\ell} \left( \frac{-c}{cz+d} \right)^ i (f_i\,\,|_{k-2i,m-i}\gamma)\right) \\
        & \stackrel{\text{(iii)}}{=} \sum_{i=0}^{\ell} (-c)^i \sum_{h=0}^n  
        \cl{D}_{n-h} (cz+d)^{-i} \cl{D}_h(f_i\,\,|_{k-2i,m-i}\gamma) \qquad \text{(apply Proposition \ref{prop:der_f_compose_g}.{\bf 1} and {\bf 3})}\\
        & = \sum_{i=0}^{\ell} (-c)^i \sum_{h=0}^n  
        { -i \choose {n-h}} c^{n-h} (cz+d)^{-i-n+h} \sum_{j=0}^h {{-k+2i-j}\choose{h-j}} \left(\frac{c}{cz+d}\right)^{h-j} (\cl{D}_jf_i)\,|_{k-2i+2j,m-i+j}\gamma \\
        & = \sum_{i=0}^{\ell} (-1)^i \sum_{j=0}^n \left(\frac{c}{cz+d}\right)^{n+i-j} (\cl{D}_jf_i)\,|_{k-2i+2j,m-i+j}\gamma\ \sum_{h=j}^n { -i \choose {n-h}} {{-k+2i-j}\choose{h-j}}.
        \end{align*}
        Now use \cite[Equation (1.5) part 2]{US} to get
        \[ \sum_{h=j}^n { -i \choose {n-h}} {{-k+2i-j}\choose{h-j}} = 
{-k+i-j\choose n-j}.\]
        Finally
        \begin{equation}\label{eq:Dnf||g} 
        \cl{D}_n(f||_{k,m}\g) = 
        \sum_{i=0}^{\ell} (-1)^i \sum_{j=0}^n \left(\frac{c}{cz+d}\right)^{n+i-j} {-k+i-j\choose n-j} (\cl{D}_jf_i)\,|_{k-2i+2j,m-i+j}\gamma.
        \end{equation}
        
        We now consider the other composition. 
        We know $\cl{D}_nf\in \tl{W}^{\,\leqslant \ell+n}_{k+2n,m+n}(\G)$
        and we have, using Proposition \ref{DerAssPol},
        \begin{align*}
        (\cl{D}_n f)||_{k+2n,m+n} \g & \stackrel{\eqref{eq:def}}{=} 
        \sum_{s=0}^{\ell+n} \left( \frac{-c}{cz+d}\right)^s (\cl{D}_nf)_s\,|_{k+2n-2s,m+n-s}\g \\
        & = \sum_{s=0}^{\ell+n} \left( \frac{-c}{cz+d}\right)^s 
        \left[\sum_{h=0}^n {n+k+h-s-1 \choose h} \cl{D}_{n-h} f_{s-h}\right]\,|_{k+2n-2s,m+n-s}\g.
        \end{align*}
        Now substitute $n-h\mapsto j$ and $s-h=s-n+j\mapsto i$,
        and recall that $f_i=0$ for all $i\not\in\{0,\dots,\ell\}$,
        to get
        \[ (\cl{D}_n f)||_{k+2n,m+n} \g= 
          \sum_{i=0}^\ell \sum_{j=0}^n \left( \frac{-c}{cz+d}\right)^{n+i-j} 
         {k+n-i-1 \choose n-j} (\cl{D}_j f_i)\,|_{k-2i+2j,m-i+j}\g. \]
        By \cite[Equation (1.5) part 1]{US}, we have
        \[ {k+n-i-1 \choose n-j} 
        = (-1)^{n-j}{-k+i-j \choose n-j}. \]
        Therefore,
        \begin{align*}
        (\cl{D}_n f)||_{k+2n,m+n} \g  
        & = \sum_{i=0}^\ell \sum_{j=0}^n \left( \frac{-c}{cz+d}\right)^{n+i-j} 
         (-1)^{n-j}{-k+i-j \choose n-j} (\cl{D}_j f_i)\,|_{k-2i+2j,m-i+j}\g \\
        & = \sum_{i=0}^\ell (-1)^i  \sum_{j=0}^n \left( \frac{c}{cz+d}\right)^{n+i-j} 
        {-k+i-j \choose n-j} (\cl{D}_j f_i)\,|_{k-2i+2j,m-i+j}\g 
        \end{align*}
        which coincides with \eqref{eq:Dnf||g}.
\end{proof}

\begin{cor}\label{c:DerfQMod}
    Let $f \in \tl{M}_{k,m}^{\,\leqslant \ell}(\G)$. Then $\cl{D}_nf \in \tl{M}_{k+2n,m+n}^{\,\leqslant \ell+n}(\G)$.
\end{cor}

\begin{proof}
    By Proposition~\ref{DerAssPol}, it suffices to show that for all $0 \leqslant j \leqslant n$, all $0 \leqslant i \leqslant \ell$, and all $\g \in \GL_2(K)$, $(\cl{D}_jf_i)\,||_{k-2i+2j,m-i+j} \g$ is holomorphic at infinity. By Proposition~\ref{prop:der_commutes_slash}, these functions may be rewritten as $\cl{D}_j(f_i\,||_{k-2i,m-i}\g)$. But if a function is holomorphic on $\O$ and at infinity, then the same is true for its hyperderivatives \cite[Theorem 3.1 and Lemma 3.6]{US}, so we are done.
\end{proof}

\subsection{Quasi-modular forms as sums of hyperderivatives}\label{s:StrDer}

Our next goal is to present a second structure theorem for Drinfeld quasi-modular forms. While in the $E$-expansion (Theorem \ref{t:StrutPolE}) 
we had a direct sum of powers of~$E$ whose coefficients were modular forms 
of appropriate weight and type, we now write a quasi-modular form as a sum 
of hyperderivatives of modular forms, thus providing an analogue of the 
main result of~\cite{KZ}. 

We will work with the operators
\[  D_n:= (-\tl{\pi})^{-n}\,\cl{D}_n\quad \text{for all }n\in \N. \]
which obviously enjoy the same properties of the $\cl{D}_n$'s (we follow the definition right 
below \cite[Equation (2)]{BP} which differs in sign from \cite[(8.5)]{G} when 
$n$ is odd). In particular, the family $\{D_n\}_{n\in\N}$ is also an iterative higher derivation on $\cl{O}$.
\smallskip

We shall use the $E$-expansion, hence the need for the following

\begin{lem}\label{DerE}
For all $n\geqslant 0$, 
we have $D_nE -E^{n+1}\in \tl{M}_{2n+2,n+1}^{\,\leqslant n-1}(\G).$
\end{lem}

\begin{proof}
Recall that $E_0=E$, $E_1=(-\widetilde\pi)^{-1}$ and 
$\mathcal{D}_i(-\widetilde\pi)^{-1}=0$ for all $i\neq 0$. Therefore, in
Proposition \ref{DerAssPol}, only the terms with $r=j-1$ and $r=j$ 
appear, and 
\begin{align*}
P_{D_n E} & = (-\widetilde\pi)^{-n}P_{\mathcal{D}_nE} 
 = (-\widetilde\pi)^{-n} \sum_{j=0}^{n+1} \left[ 
 {n\choose {j-1}} \mathcal{D}_{n-j+1}(-\widetilde\pi)^{-1} +
 {{n+1}\choose j}\mathcal{D}_{n-j}E \right]\,X^j \\
\ & = (-\widetilde\pi)^{-n-1} X^{n+1}
+{{n+1}\choose n} (-\widetilde{\pi})^{-n}EX^n 
+\sum_{j=0}^{n-1} {{n+1}\choose j} (-\widetilde{\pi})^{-j}D_{n-j}E X^j.
\end{align*}
Hence, 
\[ P_{D_n E-E^{n+1}} = P_{D_n E}-(E-\tl{\pi}^{\,-1}X)^{n+1}=
\sum_{j=0}^{n-1} {{n+1}\choose j} (-\widetilde{\pi})^{-j}
\left[ D_{n-j}E -E^{n+1-j}\right] X^j .\qedhere\]  
\end{proof}

\noindent
In particular, we know $D_1E = E^2$ by \cite[Equation (2)]{BP}. 

\smallskip
The structure as a sum of hyperderivatives 
depends on the non-vanishing of some binomial coefficients.

\noindent {\bf NVH (Non-Vanishing Hypothesis)}: let $\ell\geqslant 1$,
\begin{enumerate}[a.]
\item if $k>2\ell$, then ${k-i-1 \choose i}\not\equiv 0\pmod{p}$ for all 
indices $1\leqslant i\leqslant \ell$ such that $M_{k-2i,m-i}(\G)\neq 0$;
\item if $k=2\ell$, then ${k-i-1 \choose i}\not\equiv 0\pmod{p}$ for all 
indices $1\leqslant i\leqslant \ell-1$ such that $M_{k-2i,m-i}(\G)\neq 0$.
\end{enumerate} 
We recall that even assuming $k\equiv 2m\pmod{s(\G)}$, it is still possible that some of the $M_{k-2i,m-i}(\G)$ are zero.

\begin{thm}\label{Structure2}
For all $k,\ell\in \N$ and $m\in \Z$, we have 
\[
\tl{M}_{k,m}^{\,\leqslant \ell}(\G) =
\left\{\begin{array}{ll}
    \displaystyle{\bigoplus_{i=0}^\ell D_i M_{k-2i,m-i} (\G)} & 
    \hspace{-2.5cm}\begin{array}{l} \text{if } k > 2\ell \text{ or}\\
    \text{if } k = 2\ell \text{ and } \ell \not \equiv m \pmod{s(\G)};    
    \end{array}\\
     \ \\
    \displaystyle{\left(\bigoplus_{i=0}^{\max\{0,\ell-1\}} D_iM_{k-2i,m-i}(\G)\right) \oplus 
\langle D_{k/2-1}E\rangle_{\C_\infty}} &\quad \text{otherwise},
\end{array}\right.
\]
if and only if {\bf NVH} holds.
\end{thm}

\begin{proof}
Note that, once $k$ is fixed, the statement of {\bf NVH} is the same for $\ell=\frac{k}{2}$ and $\frac{k}{2}-1$. Moreover, 
whenever $k\geqslant 2$, the decomposition for $\ell=\frac{k}{2}\geqslant 1$ holds 
if and only if the one for $\frac{k}{2}-1$ does. Indeed, by the $E$-expansion (Theorem \ref{t:StrutPolE}), we see that the quotient space 
\[ \tl{M}_{2\ell,m}^{\,\leqslant \ell}(\G) / \tl{M}_{2\ell,m}^{\,\leqslant \ell-1}(\G) \simeq M_{0,m-\ell}(\G)E^\ell \] 
is one-dimensional if $\ell \equiv m \pmod{s(\G)}$ and zero otherwise (see \cite[Remark 5.8 (ii)]{G}).
If $\ell \equiv m\pmod{s(\G)}$, then $D_{\ell - 1}E$ is nonzero in this quotient by Lemma~\ref{DerE}. Hence the decompositions of the two spaces are equivalent, and it suffices to consider the case $k > 2\ell$.

\smallskip
\noindent ($\Longleftarrow$) Fix $k$ and~$m$. We prove the decomposition by induction on~$\ell$. If $\ell = 0$, then the statement simply reads as
\[ \tl{M}_{k,m}^{\,\leqslant 0}(\G) = D_0 M_{k,m}(\G) = M_{k,m}(\G). \]

\noindent 
Now assume $\ell > 0$. If {\bf NVH} holds for~$\ell$, then it also holds for~$\ell - 1$, and by induction we know that
\[
\tl{M}_{k,m}^{\,\leqslant \ell-1}(\G) = \bigoplus_{i=0}^{\ell-1} D_i M_{k-2i,m-i} (\G).
\]
Let $f \in \tl{M}_{k,m}^{\ell}(\G)$, then, by Proposition \ref{p:StructChange} 
and Theorem \ref{t:StrutPolE}, we can write $f$ as 
\begin{equation}\label{E-Exp-Thm-Der}  f=\sum_{i=0}^\ell f_{i,E} E^i\quad \text{with } 
f_{i,E}\in M_{k-2i,m-i}(\G) \text{ and } f_{\ell,E} \neq 0.
\end{equation}
By Corollary \ref{c:DerfQMod} we have 
$\cl{D}_\ell f_{\ell,E}\in \tl{M}_{k,m}^{\,\leqslant \ell}(\G)$ and,
since $P_{f_{\ell,E}}=(f_{\ell,E})_{0}=f_{\ell,E}$ because 
$f_{\ell,E}$ is modular, Proposition \ref{DerAssPol} reads as
\[ P_{\cl{D}_\ell f_{\ell,E}}= \sum_{j=0}^\ell 
{k-\ell-1 \choose j} \cl{D}_{\ell-j} f_{\ell,E} X^j 
 = {k-\ell-1 \choose \ell} f_{\ell,E} X^\ell
+\text{ lower degree terms}. \] 
By Proposition \ref{p:StructChange} we can write 
\[ \cl{D}_\ell f_{\ell,E}= \sum_{i=0}^\ell
(\cl{D}_\ell f_{\ell,E})_{i,E} E^i 
\quad \text{with }\, (\cl{D}_\ell f_{\ell,E})_{i,E}\in M_{k-2i,m-i}(\G). \]
Therefore, we have
\begin{align*}  P_{\cl{D}_\ell f_{\ell,E}} & =  
P_{\sum_{i=0}^\ell (\cl{D}_\ell f_{\ell,E})_{i,E} E^i }
 = \sum_{i=0}^\ell P_{(\cl{D}_\ell f_{\ell,E})_{i,E}} P_{E^i} 
  =  \sum_{i=0}^\ell (\cl{D}_\ell f_{\ell,E})_{i,E} (E-\tl{\pi}^{\,-1}X)^i\\
  & =(-\tl{\pi})^{-\ell}(\cl{D}_\ell f_{\ell,E})_{\ell,E} X^\ell
  +\text{ lower degree terms}.
\end{align*}
Since $f_{\ell,E}\neq 0$, we have $M_{k-2\ell,m-\ell}(\G)\neq 0$ and
{\bf NVH} yields ${k-\ell-1 \choose \ell}\not\equiv 0 \pmod{p}$.
Hence, $f_{\ell,E}=\a (-\tl{\pi})^{-\ell} (\cl{D}_\ell f_{\ell,E})_{\ell,E}$ 
for some $\a\in\C_\infty^*$. Therefore,
\[ \hat{f}: =f-\a D_\ell f_{\ell,E} = 
f- (-\tl{\pi})^{-\ell} \sum_{i=0}^\ell \a (\cl{D}_\ell f_{\ell,E})_{i,E} E^i =
\sum_{i=0}^{\ell-1} (f_{i,E}-(-\tl{\pi})^{-\ell}  
\a (\cl{D}_\ell f_{\ell,E})_{i,E})E^i \]
has depth $\leqslant \ell-1$. This shows that
\[ f\in \langle D_\ell f_{\ell,E}\rangle_{\C_\infty} +
\tl{M}_{k,m}^{\,\leqslant \ell-1}(\G) ,\]
which yields the desired decomposition of $\tl{M}_{k,m}^{\,\leqslant \ell}(\G)$.

\smallskip
\noindent
($\Longrightarrow$) 
Assume that {\bf NVH} does not hold. 

\noindent
Let $j$ be the minimal index 
$j\in\{ 1,\dots,\ell  \}$ such that 
$M_{k-2j,m-j} (\G)\neq 0$ and ${k-j-1\choose j}\equiv 0\pmod{p}$.
Take $g\in M_{k-2j,m-j}(\G)-\{0\}$ (so that $g=P_g=g_0$). Then, by Proposition \ref{DerAssPol},
\[ P_{\cl{D}_jg} = \sum_{s=0}^j {k-j-1\choose s} \cl{D}_{j-s}g X^s 
= \sum_{s=0}^{j-1} {k-j-1\choose s} \cl{D}_{j-s}g X^s. \]
Therefore
\[   \bigoplus_{i=0}^jD_i M_{k-2i,m-i}(\G)\subseteq 
\tl{M}_{k,m}^{\,\leqslant j-1}(\G). \]
Since $k>2j$, by hypothesis we have
\[ \bigoplus_{i=0}^j D_i M_{k-2i,m-i} (\G) =  
\tl{M}_{k,m}^{\,\leqslant j}(\G). \]
Hence $\tl{M}_{k,m}^{\, j}(\G)=0$, but this contradicts the fact that $gE^j$
is a nontrivial element of $\tl{M}_{k,m}^{\, j}(\G)$.
\end{proof}

\begin{rem}\label{r:FinDer}
The decomposition of Theorem \ref{Structure2} exists whenever
$\displaystyle{\prod_{j=1}^\ell {{k-j-1}\choose j}\not \equiv 0\pmod{p}}$
(a particular case in which ${\bf NVH}$ surely holds). There are plenty of cases in which such a product is nonzero
modulo~$p$: for example, when $\ell=1$ it reduces
to $k-2\not\equiv 0\pmod p$; when $\ell=2$, it reads as
$(k-2){{k-3}\choose 2} \not\equiv 0\pmod{p}$, and so on.
In general when $\ell$ is small with respect to~$p$, there are lots 
of weights $k$ for which the decomposition holds. 
\end{rem}

We now introduce a notation for quasi-modular forms expressed
as in Theorem \ref{Structure2}, similar to the notation for the associated polynomial
and $E$-expansion. Whenever {\bf NVH} holds, i.e.,
whenever every $f\in \widetilde{M}^{\,\leqslant\ell}_{k,m}(\G)$ is a sum of derivatives of modular
forms and of $E$, we write
\begin{equation}\label{e:StrDer} 
f= \left\{ \begin{array}{ll} 
\displaystyle{ \sum_{i=0}^\ell D_i f_{i,D} }
& \text{ if } k>2\ell ;  \\
\ & \\ 
\displaystyle{\alpha_f D_{\ell-1} E + \sum_{i=0}^{\ell-1} D_i f_{i,D}}
& \text{ if } k=2\ell ,
\end{array}\right.
\end{equation}
where $f_{i,D}\in M_{k-2i,m-i}(\G)$ for all $i$, and $\alpha_f\in \C_\infty$
(with $\a_f = 0$ if $k = 2\ell$ and $\ell \not \equiv m\pmod{s(\G)}$).

\subsection{Relation with associated polynomials and $E$-expansions}\label{s:RelStr}
Here we provide the transformation formulas to take an expression \eqref{e:StrDer} to 
its associated polynomial (the reverse not being possible in general). The relation with the $E$-expansion can be derived from the isomorphism in Theorem \ref{t:CinftyEIso}.

\begin{thm}\label{t:DerToPf}
Let $\ell\geqslant 1$ and assume $f\in \widetilde{M}^{\,\leqslant\ell}_{k,m}(\G)$ is written as in equation
\eqref{e:StrDer}. \\
Then $P_f=\displaystyle{\sum_{j=0}^\ell f_j X^j}$ where
\begin{align*} &\text{if } k>2\ell:\ 
 f_j=(-\widetilde{\pi})^{-j} 
\sum_{h=j}^\ell {{k-h-1}\choose j} D_{h-j}f_{h,D} 
\quad\text{ for all }j=0,\dots,\ell;\\ 
&\text{if } k=2\ell:\ 
 f_j=\left\{ \begin{array}{ll} \displaystyle{\alpha_f(-\widetilde{\pi})^{-\ell}}
&\text{if } j=\ell; \\
\ & \\
\displaystyle{ (-\widetilde{\pi})^{-j}\left[ \alpha_f{\ell\choose j}D_{\ell-j-1}E
+ \sum_{h=j}^{\ell-1} {{2\ell-h-1}\choose j}D_{h-j}f_{h,D}\right]} 
&\text{otherwise} . \end{array}\right.
\end{align*}
\end{thm}

\begin{proof}
Consider the case $k>2\ell$ first. Since the $f_{i,D}$ are modular, their 
associated polynomial is $P_{f_{i,D}}=(f_{i,D})_0=f_{i.D}$. In Proposition
\ref{DerAssPol} only the terms with $j=r$ appear, and we have weight $k-2i$
and depth $0$. Hence
\[ P_{D_if_{i,D}}=(-\widetilde{\pi})^{-i}P_{\mathcal{D}_if_{i,D}}
= (-\tl{\pi})^{-i} \sum_{j=0}^i {{k-i-1}\choose j}\mathcal{D}_{i-j}f_{i,D} X^j .\]
Therefore (recall that $\mathcal{D}_s\equiv 0$ when $s<0$),
\begin{align*}
P_f& = \sum_{h=0}^\ell P_{D_hf_{h,D}} =
\sum_{h=0}^\ell (-\widetilde\pi)^{-h}\sum_{j=0}^\ell {{k-h-1}\choose j}
(-\widetilde\pi)^{h-j}D_{h-j}f_{h,D} X^j \\
\ & = \sum_{j=0}^\ell (-\widetilde\pi)^{-j} 
\left[ \sum_{h=j}^\ell {{k-h-1}\choose j} D_{h-j}f_{h,D} \right]\, X^j
\end{align*}
For the case $k=2\ell$ we obtain a similar equation for the part
$\sum_{i=0}^{\ell-1}D_if_{i,D}$ and we only need to add 
$\a_f P_{D_{\ell-1}E}$.
We already computed $P_{D_nE}$ in Lemma \ref{DerE} for all 
$n\geqslant 0$. Adding that to the associated polynomial of 
$\sum_{i=0}^{\ell-1}D_if_{i,D}$, we get the final statement.   
\end{proof}

\section{Hecke operators on Drinfeld quasi-modular forms}\label{SecHecke} 
Let $\eta\in \GL_2(K)$: the $\eta$-Hecke operator on $f\in M_{k,m}(\G)$ is  defined in terms of the action of the double coset 
 $\G\eta\G$ on~$f$ (see, for example, \cite[Definition 12.11]{BBP}).
 The group $\G$ acts on the left on $\G\eta\G$ and the orbit space $\G\backslash \G\eta\G$ is a finite disjoint union $\coprod_i \G g_i$. 
 The action of $\G \eta\G$ on $f\in M_{k,m}(\G)$ is then given by 
 \[  f \mf \left( \G \eta\G\right) := \sum_i f \mf\  g_i. \]
Normalizing, the {\em $\eta$-Hecke operator} $T_\eta$ on 
$M_{k,m}(\G)$ is defined by 
\[f\longmapsto (\det \eta)^{k-m} \sum_i f\mf \ g_i.\]
The action of Hecke operators on Drinfeld modular forms has been 
thoroughly studied. Nevertheless, to our knowledge,
no studies have addressed the problem of Hecke operators acting on 
Drinfeld quasi-modular forms except for a brief mention in 
\cite[\S 4.1.1]{BP2}. In that paper, Bosser and Pellarin define 
the Hecke operators on quasi-modular forms by just applying the 
formula above for $\G=\G_0(\m)$ and $\eta=\smatrix{1}{0}{0}{\wp}$ (with $\wp$ a monic irreducible element of $\F_q[T]$ and using the representatives described in \cite[\S 7]{G}). 
They also warn the reader of the 
fact that there is no reason to believe that the output should still be a 
quasi-modular form, except for the case of depth zero. 
  
We remark that such definition is not independent of the chosen set of representatives for 
$\G_0(\m)\backslash \G_0(\m)\smatrix{1}{0}{0}{\wp}\G_0(\m)$.
Indeed, let $\{g_j\}_{j\in J}$ and $\{h_j\}_{j\in J}$ be two 
different sets of
representatives defining the $\smatrix{1}{0}{0}{\wp}$-Hecke operator. Then, 
for all $j$, there exists 
$s_j=\smatrix{a_j}{b_j}{c_j}{d_j}\in \G_0(\m)$ such that $s_jg_j=h_j$. 
Take $f\in\tl{M}^{\,\leqslant \ell}_{k,m}(\G_0(\m))$ with associated polynomial  $P_f=\sum_{i=0}^\ell f_i X^i$, then
\begin{align*}
f\longmapsto & \phantom{=} \wp^{k-m}\sum_{j\in J} f\mf h_j 
= \wp^{k-m}\sum_{j\in J} f\mf s_j g_j \\
& = \wp^{k-m} \sum_{j\in J} \left( \sum_{i=0}^\ell  f_i(z) \left( \frac{c_j}{c_j z+d_j}\right)^i \right)\mf g_j \\
& = \wp^{k-m}\sum_{j\in J}  f\mf g_j +  \wp^{k-m} \sum_{j\in J} \left( \sum_{i=1}^\ell f_i(z) \left( \frac{c_j}{c_j z+d_j}\right)^i \right)\mf g_j.
\end{align*}
There is no reason to believe that the sum 
$\displaystyle{\sum_{j\in J} \left( \sum_{i=1}^\ell  
f_i(z) \left( \frac{c_j}{c_j z+d_j}\right)^i \right)\mf g_j}$ 
is equal to zero, and in general it is not. 
For an example, take $q=p=3$, $\m=1$, $\wp=t+2$, and sets of representatives
\[ \arraycolsep=3.4pt\def\arraystretch{1.2} \mathcal{R}=\left\{ \begin{pmatrix} t+2 & 0 \\ 0 & 1 \end{pmatrix}\,,\ 
\begin{pmatrix} 1 & 0 \\ 0 & t+2 \end{pmatrix}\,,\
\begin{pmatrix} 1 & 1 \\ 0 & t+2 \end{pmatrix}\,,\
\begin{pmatrix} 1 & 2 \\ 0 & t+2 \end{pmatrix}\, \right\}\ 
\text{ and }\ 
\mathcal{R}'=\begin{pmatrix}1&1\\t&t+1\end{pmatrix}\mathcal{R}.\]
Consider $E\in \widetilde{M}^{\,\leqslant 1}_{2,1}(\GL_2(\F_3[T]))$ with
$E_1=(-\tl{\pi})^{-1}$, then the ``extra'' part above is the sum 
\[ \sum_{\eta\in \mathcal{R}'} (E_1 |_{2,1} \eta)(z) = 
\sum_{\eta\in \mathcal{R}} (-\frac{1}{\widetilde{\pi}}\cdot
\frac{t}{tz+t+1} |_{2,1} \eta)(z), \]
which, evaluated at $z=0$, gives
\[ -\frac{t(t+2)}{\tl{\pi}} \left[ \frac{t^6+\text{lower degree 
terms}}{(t+1)(t+2)^2(t^2+t+2)(t^2+2t+2)}\right] \neq 0.\]

\subsection{Double-slash and Hecke operators}
In order to deal with this problem, we apply the double-slash operator as follows.

\begin{defin}\label{HeckeQuasi} Let $\eta\in \GL_2(K)$, and let 
$\mathscr{R}(\G,\eta)$ be a (finite) set of representatives for the orbit space 
 $\G\backslash\G\eta\G$. Then the {\em $\eta$-Hecke operator} $T_\eta$ is defined on
 $f\in \tl{M}^{\,\leqslant \ell}_{k,m} (\G)$ by
\[ f \longmapsto (\det \eta)^{k-m} \sum_{\g\in\mathscr{R}(\G,\eta)} f\mmf \gamma . \]
\end{defin}

\begin{rem}\label{RemDouble}\ 
Parts {\bf 1}, {\bf 3} and {\bf 4} of Proposition~\ref{thm:double-slash-properties} show 
that the Hecke operators in Definition~\ref{HeckeQuasi} are 
well-defined, i.e., independent of the chosen set of representatives. 
\end{rem}

In terms of associated polynomials, we define $\mathcal T_\eta : \cl{P}^{\,\leqslant \ell}_{k,m}(\G) \to \cl{P}^{\,\leqslant \ell}_{k,m}(\G)$ by
\[
\mathcal T_\eta(P(z,X)) = (\det \eta)^{k-m} \sum_{\g\in\mathscr{R}(\G,\eta)} (P \mmf \g) (z,X),
\]
(from now on we shall simply write $\sum_\g$ to denote the sum from the equation above).

\begin{prop}\label{t:UpCommPf}
Let $f\in \tl{M}^{\leqslant \ell}_{k,m}(\G)$ with associated polynomial $P_f = \sum_{i=0}^{\ell} f_iX^i$. Then we have for all $0 \leqslant i \leqslant \ell$,
\[ (T_\eta(f))_i = (\det \eta)^i T_\eta (f_i) .\]
In particular, $T_\eta (f)\in \tl{M}^{\leqslant \ell}_{k,m}(\G)$. Moreover, its depth is $\leqslant \ell-1$ if and only if $f_\ell\in \Ker T_\eta$.
\end{prop}

\begin{proof}
As a direct consequence of Corollary~\ref{cor:restriction} and Proposition \ref{thm:double-slash-properties}.{\bf 2}, we have that the diagram
\begin{equation}\label{eq:CommDiag} \xymatrix{
\tl{M}^{\,\leqslant \ell}_{k,m}(\G) \ar[d] \ar[rr]^{T_\eta} & &\tl{M}^{\,\leqslant \ell}_{k,m}(\G)  \ar[d] \\
\cl{P}^{\,\leqslant \ell}_{k,m}(\G) \ar[rr]^{\cl{T}_\eta} 
& & \cl{P}^{\,\leqslant \ell}_{k,m}(\G)  }
\end{equation}
commutes. Indeed
\[ \cl{T}_\eta (P_f) = (\det \eta)^{k-m} \sum_\g P_f \mmf \g = 
(\det \eta)^{k-m} \sum_\g P_{f\mmf \g} = P_{(\det \eta)^{k-m}\sum_\g f\mmf \g}
= P_{T_\eta (f)}. \]
Moreover, by Proposition \ref{thm:double-slash-properties}.{\bf 6},
\[ \cl{T}_\eta (P_f) = (\det \eta)^{k-m} \sum_\g P_f \mmf \g = 
(\det \eta)^{k-m} \sum_{i=0}^\ell \sum_\g (f_i||_{k-2i,m-i} \g) X^i
=\sum_{i=0}^\ell ((\det \eta)^i\, T_\eta (f_i)) X^i. \]
Comparing the two formulas, we get $(T_\eta (f))_i=(\det \eta)^i\, T_\eta (f_i)$.
\end{proof}

\begin{rem}\label{r:Kernels}
We know that $f_\ell$ is modular and also equal to 
$(-\tl{\pi})^{-\ell} f_{\ell,E}$. For $\eta_\wp=\smatrix{1}{0}{0}{\wp}$, there are some known results (or conjectures) on $\Ker T_{\eta_\wp}$ 
for modular forms. We recall and generalize them to 
quasi-modular forms in Section \ref{s:ExplicitHecke} (see, in
particular Proposition \ref{p:KerUp} and Remark \ref{r:KerTp}). 
\end{rem}

An immediate consequence is the following characterization of
eigenforms.

\begin{cor}\label{c:EigenDer}
Let $f\in \tl{M}^{\,\ell}_{k,m}(\G)$ with associated polynomial $P_f=\sum_{i=0}^\ell f_iX^i$, and let $\lambda \in \C_\infty$.
Then $T_\eta(f)=\lambda f$ if and only $T_\eta (f_i)=\frac{\lambda}{(\det \eta)^i}\,f_i$ for all $0\leqslant i\leqslant \ell$.

\noindent
In other words, $f$ is a $T_\eta$-eigenform of eigenvalue $\lambda$
if and only if every $f_i$ is a quasi-modular $T_\eta$-eigenform (possibly zero) of eigenvalue~$\frac{\lambda}{(\det \eta)^i}$ for all $0\leqslant i\leqslant \ell$.
\end{cor}

\begin{proof}
Just note that $f$ is a $T_\eta$-eigenform if and only if
\[ \lambda \sum_{i=0}^\ell f_i X^i = P_{\lambda f} = 
P_{T_\eta (f)} = \sum_{i=0}^\ell (\det \eta)^i\, T_\eta (f_i)  X^i . \qedhere\]
\end{proof}

\begin{rem}\label{r:EigVal}
Since $f_\ell$ is a modular form, the previous corollary shows that the {\em possible} eigenvalues of quasi-modular forms of 
weight $k$, type $m$ and depth $\ell$ are completely determined by the eigenvalues of modular
forms of weight $k-2\ell$ and type $m-\ell$.
Conversely, Proposition~\ref{p:DerHecke} below shows that if $f \in M_{k-2\ell,m-\ell}(\G)$ is an eigenform with $T_{\eta}$-eigenvalue $\lambda$ and $\cl{D}_{\ell}f \neq 0$, then $\cl{D}_{\ell}(f)$ is a quasi-modular eigenform in $\tl{M}_{k,m}^{\,\leqslant\ell}(\G)$ with $T_{\eta}$-eigenvalue $(\det \eta)^{\ell} \lambda$; however, the depth of $\cl{D}_{\ell}f$ may be strictly smaller than $\ell$.
\end{rem}
 
\subsection{Hyperderivatives and Hecke operators}

The interaction of hyperderivatives and Hecke operators is described by the following Proposition (cf.~\cite[Lemma~4.6]{BP2}).

\begin{prop}\label{p:DerHecke}
    Let $f \in \tl{M}_{k,m}^{\,\leqslant \ell}(\G)$ and $\eta \in \GL_2(K)$. Then for any $n \geqslant 0$, 
    \[ T_{\eta} (\cl{D}_n f) =  (\det \eta)^{n} \cl{D}_n (T_{\eta} f). \]
    In particular, if $f$ is a $T_\eta$-eigenform of eigenvalue
    $\lambda$ and $\cl{D}_n f\neq 0$, then $\cl{D}_n f$ is a $T_\eta$-eigenform of eigenvalue $(\det \eta)^n\lambda$.
\end{prop}

\begin{proof}
    By definition of $T_{\eta}$ and the linearity of $\cl{D}_n$, we have
    \[
    \cl{D}_n (T_{\eta} f) = (\det \eta)^{k-m} \sum_{\g} \cl{D}_n(f\mmf \g).
    \]
    Applying Proposition \ref{prop:der_commutes_slash}, this equals $(\det \eta)^{-n} T_{\eta}(\cl{D}_n f)$, as desired.
\end{proof}

\begin{rem}\label{r:Hecke+Der}
By Proposition~\ref{p:DerHecke}, the decomposition of Theorem \ref{Structure2} is Hecke equivariant up to a character.
In particular, if $f=\sum_{i=0}^\ell D_i f_{i,D}\in \tl{M}^{\,\leqslant \ell}_{k,m}(\G)$, then
\[ T_\eta (f) = \sum_{i=0}^\ell T_\eta( D_i f_{i,D}) =
\sum_{i=0}^\ell  (\det \eta)^i D_i(T_\eta (f_{i,D})) \]
(an analogous formula holds for the case $k=2\ell$).
Hence, the action of Hecke operators on $\tl{M}^{\,\leqslant \ell}_{k,m}(\G)$ is completely determined by their action on the
modular forms $M_{k-2i,m-i}(\G)$ for $i=0,\dots,\ell$ whenever 
{\bf NVH} holds. For instance, in this case $T_{\eta}$ is diagonalisable on $\tl{M}_{k,m}^{\,\leqslant \ell}(\G)$ if and only if $T_{\eta}$ is diagonalisable on $M_{k-2i,m-i}(\G)$ for $i = 0,\dots,\ell$.

\noindent
Since {\bf NVH} is a necessary condition for the decomposition to exist, it is unclear whether such statements are still true when {\bf NVH} does not hold. This is in contrast with the classical setting, where a decomposition as in Theorem~\ref{Structure2} always exists.
It would be interesting to understand how Hecke operators (and other 
operators which behave well with respect to hyperderivations)
act when {\bf NVH} does not hold. 
\end{rem}

\subsection{The case  $\G_0(\m)$}\label{s:ExplicitHecke}
From now on we restrict to  
$\G_0(\m)$ for some ideal $\m$ of $A$.

\noindent 
Let $\m$, $\p=(\wp)$ be two ideals of $A$ with $\p$ prime and $\wp$ monic.
We recall that a set of representatives for 
 $\G_0(\mathfrak{m})\backslash \G_0(\mathfrak{m})\smatrix{1}{0}{0}{\wp}\G_0(\mathfrak{m})$ is 
 provided by matrices $\smatrix{a}{b}{0}{d}$ with $a,d\in \F_q[T]$ monic, such that $(ad)=\p$ and $(a)+\mathfrak{m}=\F_q[T]$, and   $b$ varies in a set of representatives for $\F_q[T]/(d)$
 (see \cite{Ar}).
  
In this setting we put 
\[ T_{\smatrix{1}{0}{0}{\wp}}:=\left\{ \begin{array}{ll} T_\mathfrak{p} & \text{ if } \mathfrak{p}\nmid\mathfrak{m}; \\
U_\mathfrak{p} & \text{ if } \mathfrak{p} \mid \mathfrak{m}
\end{array}\right. \]
(it is easy to see that the definition is independent of the choice of the generator $\wp$). Our formulas for the 
{\em $\p$-Hecke operators} will be
\begin{align*}
T_\p(f)(z) & =  \wp^{k-m}\left[ (f\mmf \smatrix{\wp}{0}{0}{1})(z) +
 \sum_{\begin{subarray}{c}  Q\in \F_q[T]\\ \deg Q< \deg \wp \end{subarray}}\hspace{-.3cm} (f\mmf \smatrix{1}{Q}{0}{\wp})(z)\right] \quad & \text{if } \p\nmid \m; \\
  U_\p (f)(z)&  =\wp^{k-m} \sum_{\begin{subarray}{c}  Q\in \F_q[T]\\ \deg Q< \deg \wp \end{subarray}}\hspace{-.3cm} (f\mmf \smatrix{1}{Q}{0}{\wp})(z)\quad & \text{if } \p \mid \m 
\end{align*}
(we shall simply write $\sum_{Q}$ for this set of representatives).
 When $\p\mid\m$, the $\p$-Hecke operator $U_\p$ is also known as 
Atkin-Lehner operator.

\begin{rem}\label{r:RemDoubleSl}\ 
For $\G=\G_0(\m)$, with the above representatives (as mentioned in 
Remark \ref{r:Double1}) there is no difference between the slash 
and the double-slash operators, so one can actually perform all 
computations forgetting the double-slash (i.e., using the definition 
of \cite[\S 4.1.1]{BP2}). Just keep in mind that to get the same result 
with a different set of representatives (in which not all the elements 
in position $(2,1)$ are zero) the double-slash has to come into play.
\end{rem}

To avoid ambiguity between $T_\p$ and $U_\p$, we assume 
$\m$ is a nonzero ideal of $A$ relatively prime to $\p$.
To complete the description of Hecke operators, we shall compute the action of $U_\p$ (respectively, $T_\p$) on the $E$-expansion of quasi-modular forms 
of level $\m\p$ (respectively, of level $\m$). 
By the existence and uniqueness of the $E$-expansion and the
linearity of Hecke operators it is enough to check the action on forms of 
type $f E^n$, with $f$ a Drinfeld modular form of weight $k-2n$ and type $m-n$.

\subsubsection{The degeneracy map $\delta_\p$}\label{s:DegeMap}
It is possible to lower or  raise the level of a Drinfeld modular
form using {\em trace maps} and {\em degeneracy maps}, respectively. For details the reader is referred to~\cite{BV} and~\cite{V}. The same holds for Drinfeld quasi-modular forms. 

\noindent
The degeneracy maps are defined as
\begin{align} \d_1:   \tl{M}_{k,m}^{\,\leqslant \ell} (\G_0(\m)) & 
\longrightarrow \tl{M}_{k,m}^{\,\leqslant \ell} (\G_0(\m\p))\\\nonumber
 f(z) & \longmapsto f(z)\,,\\
  \d_\p:   \tl{M}_{k,m}^{\,\leqslant \ell} (\G_0(\m)) & \longrightarrow 
  \tl{M}_{k,m}^{\,\leqslant \ell} (\G_0(\m\p))\\\nonumber
 f(z) & \longmapsto \wp^{-m}\left(f\mmf \begin{pmatrix} \wp&0\\0&1\end{pmatrix}\right) (z)=f(\wp z).
\end{align}
Note that:\begin{itemize} 
\item $||_{k,m}$ and $|_{k,m}$ are equivalent here;
\item the map $\delta_\p$ has a normalization different from 
the one in~\cite{BV} and~\cite{V}, hence the formulas from those papers have to be adapted a bit.
\end{itemize}

\noindent
In what follows we shall mainly work with $U_\p$ in level
$\m\p$ and then obtain formulas for $T_\p$ in level $\m$
using the formal equality
\[ T_\p = \wp^k \delta_\p + U_\p .\]
Since $U_\p$ is defined on forms of 
level $\m\p$ there is a little abuse of notation in the use of $U_\p$ here, 
but we implicitly use the natural inclusion $\delta_1$.

A crucial role is played by the {\em $\p$-stabilization}
of $E$
\[ E_\p:=E-\wp\delta_\p E .\]
It is well-known, see for example \cite[\S 1.1.5]{P}, that $E_\p\in S_{2,1}(\G_0(\p))$
and that $T_{\smatrix{1}{0}{0}{\mathscr{Q}}} E_\p =\mathscr{Q} E_\p$ for all nonzero primes $(\mathscr{Q})=\mathfrak q\neq \p$.
In particular, we have that $\delta_\p E\in \widetilde{M}^{\,1}_{2,1}(\G_0(\p))$ with associated
polynomial
\[ P_{\delta_\p E}=\wp^{-1}(P_E-P_{E_\p})=
\delta_\p E -(\wp\tl{\pi})^{-1}X. \]

We now show that $\delta_\p$ is well-defined, i.e., it 
transforms a quasi-modular form of depth $\ell$ and level $\m$ into 
a quasi-modular form of depth $\ell$ and level $\m\p$ (as it does 
for Drinfeld modular forms, see \cite[Section 2]{BV}).

\begin{prop}\label{p:deltapf}
Let $f \in  \tl{M}_{k,m}^{\,\leqslant \ell} (\G_0(\m))$. Then $\delta_\p f\in \tl{M}_{k,m}^{\,\leqslant \ell} (\G_0(\m\p))$.

\noindent
Moreover, $\text{depth}\,f=\text{depth}\,\delta_\p f$. 
\end{prop}

\begin{proof}
Let $f=\sum_{i=0}^\ell f_{i,E} E^i \in  \tl{M}_{k,m}^{\,\leqslant \ell} (\G_0(\m))$ be the $E$-expansion of $f$, with $f_{i,E}\in M_{k-2i,m-i}(\G_0(\m))$ for all $i$. Then,
\begin{align*}
\delta_\p f & = \sum_{i=0}^\ell 
\wp^{-m}\left(f_{i,E} |_{k-2i,m-i} 
\smatrix{\wp}{0}{0}{1}\right)\left(E^i |_{2i,i} 
\smatrix{\wp}{0}{0}{1}\right)
 = \sum_{i=0}^\ell \delta_\p f_{i,E} \left(\delta_\p E \right)^i\\
 & = \sum_{i=0}^\ell \delta_\p f_{i,E} \sum_{j=0}^i { i\choose j} \wp^{-i}(-E_\p)^{i-j} E^{j}
 = \sum_{j=0}^\ell \sum_{i=j}^\ell { i\choose j} \wp^{-i}\delta_\p f_{i,E} (-E_\p)^{i-j} E^{j}.
 \end{align*}
Therefore, the coefficient of $E^j$ is 
\begin{equation}\label{e:EqDeltapfCoeff} (\delta_\p f)_{j,E} =  
\sum_{i=j}^\ell {i \choose j}\wp^{-i} (-E_\p)^{i-j} \delta_\p  f_{i,E}= \sum_{s=0}^{\ell-j}  { s+j \choose j}\wp^{-s-j}(-E_\p)^s \delta_\p f_{s+j,E}
 \end{equation}
which is a modular form in $M_{k-2j,m-j} (\G_0(\m\p))$. 
Hence we have obtained
the $E$-expansion of $\delta_\p f$ and the first claim follows.

\noindent
For the final statement, just note that 
$(\delta_\p f)_{\ell,E}= \wp^{-\ell}\delta_\p f_{\ell,E}$ and, by
definition, $(\delta_\p g)(z)=g(\wp z)$. 
Hence, the map $\delta_\p$ is obviously injective. 
\end{proof}

Finally, we observe that $\text{Im}\, \delta_\p \subseteq \Ker U_\p$ (as in \cite[Equation~(4)]{BV}
for Drinfeld modular forms). Indeed, 
\begin{equation}\label{e:Updeltap=0}
\wp^{m-k} U_\p(\delta_\p f)  = 
\sum_Q (\delta_\p f \mmf \smatrix{1}{Q}{0}{\wp}) 
=\sum_Q \wp^{-m}(f\mf \smatrix{1}{Q}{0}{1} \smatrix{\wp}{0}{0}{\wp}) 
 = \wp^{m-k}\sum_Q f =0, 
\end{equation}
because there are $q^{\deg \wp}$ representatives $Q$.

\subsubsection{Atkin-Lehner operator at level $\m\p$}
We are now ready to describe the action of $U_\p$ on quasi-modular forms.
 
\begin{lem}\label{UpEn}
Let $f\in M_{k-2n,m-n}(\G_0(\m\p))$. Then
\[  U_\p(f E^n)= U_\p (f E_\p^n)- \sum_{h=1}^n { n \choose h} 
(-\wp)^h U_\p(f E^{n-h}) E^h. \]
\end{lem}
 
 \begin{proof} We compute $U_\p(fE_\p^n)$ (which is actually
 a Drinfeld modular form):
 \begin{align*}
 \wp&^{m-k}U_\p(f E_\p^n) 
 = \sum_Q (f E_\p^n \,||_{k,m} \smatrix{1}{Q}{0}{\wp}) \\
 & = \sum_Q (f \sum_{h=0}^n {n \choose h}(-\wp)^h E^{n-h}
 \delta_\p E^h \,||_{k,m} \smatrix{1}{Q}{0}{\wp})\\
 & = \wp^{m-k}U_\p(fE^n)+
 \sum_{h=1}^n {n \choose h}(-\wp)^h
 \sum_Q(f E^{n-h}||_{k-2h,m-h} \smatrix{1}{Q}{0}{\wp})
 (\delta_\p E\,||_{2,1} \smatrix{1}{Q}{0}{\wp})^h.
 \end{align*}
 Since
 \begin{equation}\label{eq:deltapE} 
 \delta_\p E\,||_{2,1} \smatrix{1}{Q}{0}{\wp}=
 \wp^{-1} E\,||_{2,1} \smatrix{1}{Q}{0}{1}\smatrix{\wp}{0}{0}{\wp}= \wp^{-1}E, 
 \end{equation}
 we get
 \begin{align*}
 \wp^{m-k}U_\p(f E_\p^n)& = \wp^{m-k}U_\p(fE^n) + \sum_{h=1}^n {n \choose h}(-E)^h 
 \sum_Q (f E^{n-h}\,||_{k-2h,m-h} \smatrix{1}{Q}{0}{\wp} )\\
 & = \wp^{m-k}\left[ U_\p(fE^n) +\sum_{h=1}^n {n\choose h} (-\wp)^h U_\p(f E^{n-h}) E^h \right]. 
 \end{align*}
The lemma follows.
 \end{proof}

\noindent The formula of the lemma does not fit the description of 
the $E$-expansion, because the coefficients of the powers of $E$ are not 
modular forms in general. This is the goal of the next theorem.  

\begin{thm}\label{t:HeckeAct}
Let $f\in M_{k-2n,m-n}(\G_0(\m\p))$. Then  
\[  U_\p(f E^n)= \sum_{h=0}^n {n \choose h} \wp^h U_\p (f E_\p^{n-h})E^h. \]
\end{thm}

\begin{proof}
The claim holds for $n=0,1$ by Lemma \ref{UpEn}.

\noindent
We proceed by induction. Assume the claim holds for powers of $E$ strictly smaller than $n$. We have (using the inductive step in the second line)
\begin{align*}
U_\p (f E^n) & = U_\p(f E_\p^n) - 
\sum_{i=1}^n { n \choose i} (-1)^i \wp^i U_\p(f E^{n-i}) E^i \\
& = U_\p(f E_\p^n) - \sum_{i=1}^n (-1)^i  \sum_{j=0}^{n-i} { n \choose i} 
{ n-i \choose j}  \wp^{i+j} U_\p(f E_\p^{n-i-j}) E^{i+j} \\
& = U_\p(f E_\p^n) - \sum_{i=1}^n (-1)^i \sum_{h=i}^n {n \choose h}{h \choose i} \wp^h U_\p(f E_\p^{n-h}) E^h \\ 
& = U_\p(f E_\p^n) -\sum_{h=1}^n {n \choose h} \left[ \sum_{i=1}^h (-1)^i {h \choose i} \right] \wp^h U_\p(f E_\p^{n-h}) E^h   \\
 & = U_\p(f E_\p^n)+\sum_{h=1}^n {n \choose h} \wp^h U_\p(f E_\p^{n-h}) E^h 
 = \sum_{h=0}^n {n \choose h} \wp^h U_\p (f E_\p^{n-h})E^h. \qedhere
\end{align*}
\end{proof}

\begin{cor}\label{c:UpQuasiMod}
Let $f=\sum_{i=0}^\ell f_{i,E} E^i \in \widetilde{M}_{k,m}^{\,\leqslant \ell}(\G_0(\m\p))$. Then
\[ (U_\p(f))_{i,E} = \wp^i\sum_{h=0}^{\ell-i} {{h+i} \choose i} 
U_\p(f_{h+i,E}E_\p^h) \quad \text{ for all }\,i=0,\dots,\ell .\]
\end{cor}

\begin{proof}
Just plug the formula of Theorem \ref{t:HeckeAct} into
$U_\p(f)=\sum_{i=0}^\ell U_\p(f_{i,E}E^i)$.
\end{proof}

The following generalizes \cite[Theorem 2.8 and Proposition 2.13]{BV}.
\begin{prop}\label{p:KerUp}
Let $f\in \widetilde{M}^{\,\ell}_{k,m}(\G_0(\m\p))$. Then $f\in \Ker U_\p$ if and only if 
$f=\wp^m\delta_\p g$ for some $g\in \widetilde{M}^{\,\ell}_{k,m}(\G_0(\m))$.  
\end{prop}

\begin{proof} 
We have already seen in equation \eqref{e:Updeltap=0} that $\text{Im}\, \delta_\p \subseteq \Ker U_\p$. Moreover, the statement holds for $\ell=0$ by
 \cite[Proposition 2.13]{BV} (paying attention to the different normalization for $\delta_\p$).

\noindent
Let $f=\sum_{i=0}^\ell f_{i,E}E^i\in \Ker U_\p$. Then, by Corollary \ref{c:UpQuasiMod},
we have
\[ \sum_{h=0}^{\ell-i} {{h+i} \choose i} 
U_\p(f_{h+i,E}E_\p^h)=0 \quad \text{ for all }\,i=0,\dots,\ell .\]
Hence, for all $i=0,\dots,\ell$, there exists 
a modular form $g_{i,E}\in M_{k-2i,m-i}(\G_0(\m))$ such that
\[ \sum_{h=0}^{\ell-i} {{h+i} \choose i} 
f_{h+i,E}E_\p^h =\wp^{m-i}\delta_\p g_{i,E}\,.\]
Let $g:=\sum_{i=0}^\ell g_{i,E} E^i \in \widetilde{M}^{\,\ell}_{k,m}(\G_0(\m))$. Then, by (the proof of) Proposition \ref{p:deltapf},
\[ \delta_\p g = \sum_{h=0}^\ell  \left( \sum_{i=h}^\ell 
{i\choose h}\wp^{-i}\delta_\p g_{i,E} (-E_\p)^{i-h}\right) E^h. \]
Substituting we get
\begin{align*}
\wp^m\delta_\p g & = \sum_{h=0}^\ell  \left( \sum_{i=h}^\ell 
{i\choose h}\left(\sum_{s=0}^{\ell-i} {{s+i} \choose i} f_{s+i,E}E_\p^s \right) (-E_\p)^{i-h}\right) E^h \\
 & = \sum_{h=0}^\ell  \left( \sum_{i=h}^\ell 
{i\choose h}\left(\sum_{j=i}^\ell {j \choose i} f_{j,E}E_\p^{j-i} \right) (-E_\p)^{i-h}\right) E^h \\
& = \sum_{h=0}^\ell  \left( \sum_{i=0}^\ell \sum_{j=0}^\ell 
{{j-h}\choose {i-h}}{j \choose h} f_{j,E}E_\p^{j-h}(-1)^{i-h}\right) E^h\\
& = \sum_{h=0}^\ell  \left( \sum_{j=h}^\ell {j\choose h} f_{j,E}E_\p^{j-h} 
\left(\sum_{i=h}^j {{j-h}\choose {i-h}}(-1)^{i-h}\right)\right) E^h 
 =\sum_{h=0}^\ell f_{h,E} E^h= f. \qedhere
\end{align*}
 \end{proof}

\subsubsection{The $\p$-Hecke operator at level $\m$}
To compute the action of $T_\p$, we use the formal equality
\[ T_\p = \wp^k \delta_\p + U_\p .\]

\begin{thm}\label{t:HeckeActBis}
Let $f\in M_{k-2n,m-n}(\G_0(\m))$. Then the $E$-expansion of $T_\p(fE^n)$ is 
\[  T_\p(fE^n)=   
\sum_{h=0}^n {n\choose h} \wp^h \left[ \wp^{k-n-h} 
\delta_\p f (-E_\p)^{n-h} + U_\p(fE_\p^{n-h})\right] E^h .  \]
\end{thm}

\begin{proof} We have
\begin{align*}
T_\p(fE^n)&=\wp^k\delta_\p(fE^n)+U_\p(fE^n)=
\wp^k\delta_\p f(\delta_\p E)^n+U_\p(fE^n) \\
 & \stackrel{\text{Th. } \ref{t:HeckeAct}}{=} \wp^{k-n}\delta_\p f (E-E_\p)^n+
\sum_{h=0}^n {n\choose h} \wp^h U_\p(fE_\p^{n-h})E^h  
\\
\ & = \wp^{k-n}\delta_\p f \sum_{h=0}^n {n\choose h} (-E_\p)^{n-h}E^h 
+\sum_{h=0}^n {n\choose h} \wp^h U_\p(fE_\p^{n-h})E^h \\
\ &= \sum_{h=0}^n {n\choose h} \wp^h \left[ \wp^{k-n-h} \delta_\p f (-E_\p)^{n-h}
+ U_\p(fE_\p^{n-h})\right] E^h .\qedhere
\end{align*} 
\end{proof}

\noindent

\begin{rem}\label{r:KerTp}
Since $(T_\p(f))_{\ell,E}=\wp^\ell T_\p(f_{\ell,E})$, 
the depth of $T_\p(f)$ is $<\ell$ if and only if $f_{\ell,E}\in 
\Ker T_\p$. We recall that, in analogy with \cite[Conjecture~1.1]{BVexp}, we expect 
$\Ker T_\p=0$ when the level $\m$ is prime (see \cite[Theorem~3.1]{BVBams} and 
\cite[Theorem~4.8]{DKBams} for some special cases). A general proof for $\m=(1)$ has recently
been provided in~\cite{dV}.
\end{rem}

We end this section by showing that nonzero $T_\p$-eigenvalues
lift to $U_\p$-eigenvalues (as they do for modular forms, see \cite[Equation~(5)]{BV}).

\begin{lem}\label{l:LiftEigen}
Let $f\in \tl{M}^{\,\ell}_{k,m}(\G_0(\m))$ and $\lambda\in\C_\infty^*$. Then,
\[ T_\p(f)=\lambda f\ \text{ if and only if }\ 
 U_\p\left( f-\frac{\wp^k}{\lambda}\cdot \delta_\p f\right)=
\lambda\left(f-\frac{\wp^k}{\lambda}\cdot \delta_\p f\right).\]
\end{lem}

\begin{proof} It suffices to note that, by Proposition \ref{p:KerUp},
\[ U_\p\left( f-\frac{\wp^k}{\lambda}\cdot \delta_\p f\right) =
U_\p(f)=T_\p(f)-\wp^k\delta_\p f .\qedhere\]
\end{proof}

\end{document}